\documentclass[journal,transmag]{IEEEtran}

\usepackage[utf8]{inputenc} 
\usepackage[T1]{fontenc}    
\usepackage{hyperref}       
\usepackage{url}            
\usepackage{booktabs}       
\usepackage{amsfonts}       
\usepackage{nicefrac}       
\usepackage{microtype}      
\usepackage{xcolor}         
\usepackage{graphicx}
\usepackage{amsmath}
\usepackage{bm}
\usepackage{bbm}
\usepackage{dsfont}
\usepackage{xcolor}
\usepackage{enumerate}

\usepackage{longtable}
\usepackage{pifont}  
\usepackage{bm}
\usepackage{bbm}
\usepackage{dsfont}

\usepackage{amsmath, amssymb, amsfonts}  

\usepackage[mathscr]{eucal}
\usepackage{subfigure}
\usepackage{theorem}

\usepackage{algorithm}

\newtheorem{theorem}{Theorem}

\newtheorem{proposition}{Proposition}
\newtheorem{lemma}{Lemma}

\newtheorem{assumption}{Assumption}

{\theorembodyfont{\upshape}
\newtheorem{remark}{Remark}

}

\def\qed{\hfill \vrule height 5pt width 5pt depth 0pt \medskip}
\newcommand{\proof}{\noindent {\bf Proof. }}

\def\diag{{\rm diag}}
\def\sign{{\rm sgn}}

\def\rank{{\rm rank}}

\newcommand{\bp}{\bm{p}}
\newcommand{\bq}{\bm{q}}
\newcommand{\br}{\bm{r}}
\newcommand{\bs}{\bm{s}}
\newcommand{\bu}{\bm{u}}
\newcommand{\bv}{\bm{v}}
\newcommand{\bx}{\bm{x}}
\newcommand{\by}{\bm{y}}
\newcommand{\bw}{\bm{w}}
\newcommand{\bz}{\bm{z}}

\newcommand{\bK}{{F_K}}
\newcommand{\bQ}{{F_Q}}
\newcommand{\bV}{{F_V}}

\newcommand{\beq}{\begin{equation}}
\newcommand{\eeq}{\end{equation}}
\newcommand{\beqa}{\begin{eqnarray}}
\newcommand{\eeqa}{\end{eqnarray}}
\newcommand{\beqan}{\begin{eqnarray*}}
\newcommand{\eeqan}{\end{eqnarray*}}

\newcommand{\pde}[2]{ \frac{\partial #1}{\partial #2} }

\newcommand{\bite}{\begin{itemize}}
\newcommand{\eite}{\end{itemize}}
\newcommand{\benu}{\begin{enumerate}}
\newcommand{\eenu}{\end{enumerate}}

\definecolor{darkpastelgreen}{rgb}{0.01, 0.75, 0.24}

\title{Asymptotic Analysis of Discrete-Time Self-Attention Dynamics in Transformers \thanks{C. Altafini is with the Division of Automatic Control, Dept. of Electrical Engineering, Link\"oping University, SE-58183, Link\"oping, Sweden. email: {\tt\small claudio.altafini@liu.se}. Work supported in part by the Swedish Research Council (grant n. 2024-04772) and by the Wallenberg AI, Autonomous Systems and Software Program (WASP) funded by the Knut and Alice Wallenberg Foundation.}
}

\author{Claudio Altafini}

\begin{document}

\maketitle
%

\begin{abstract}
In machine learning, a self-attention dynamics describes the attention mechanism of transformers, and is obtained treating the layer index of a transformer as time index. 
Even though naturally discrete-time, this dynamical system is most often studied in continuous-time.
Scope of this paper is to provide a thorough investigation of the discrete-time version of the self-attention dynamics: a complete classification of its equilibria is given, and their stability properties are studied, in both the single-head and the multihead case. We show that in both cases the main attractor of the dynamics is a consensus equilibrium aligned with the principal eigenvector of the value matrix (in the single-head case, a corresponding matrix of weights in the multihead case), even though other locally asymptotically stable bipartite equilibria can be show to exist at least in the single-head situation. 
The behavior is confirmed by looking at the learned weight matrices of a real transformer.

\end{abstract}


\section{Introduction}

The core algorithm inside a large language model is a so-called attention-based transformer \cite{vaswani2017attention}.
This architecture, introduced in 2017, has rapidly become the standard in many machine learning problems, also beyond natural language processing. 
A transformer is composed of multiple layers having identical structure, acting on ``tokens'', that is, vectorial representations of the elementary units in which a set of data can be decomposed (words for language processing, patches of pixels for image processing, etc.). 
Inside a layer are two main subblocks: one is the attention mechanism and the other a feedforward neural network (FFN).
The attention mechanism quantifies how much importance each token is giving to the other tokens, which justify why it is normally referred to as ``self-attention'' (the tokens' attention is a function of the tokens themselves).
It is implemented as a softmax function of an inner product involving the tokens multiplied by two learned weight matrices called query and key.
The attention coefficients provide the weights of a convex combination of the tokens, which is in turn multiplied by a third matrix of weights called the value matrix. 
In a so-called multihead transformer, the attention mechanism inside a layer is ``multiplexed'' in such a way that different parts of a token can use different attention weights.
The results at the various heads are then ``concatenated'' using a fourth matrix of weights called the output matrix. 
The FFN subblock of a layer is instead normally composed of a 2-layer ``perceptron'', that is, a static map containing two learned weight matrices and bias terms with in the middle an activation function. 
Other ingredients of a transformer layer are residual connections, in which to the token value exiting one of the two subblocks above is added the value of the token before the subblock, and normalization operations, which we simplify here as $ L_2 $ normalizations of the token vectors, in place of the original ``layernorm'' operation (an affine, elementwise operation).

A popular approach to build mathematical models of transformers is to interpret the layer index as time axis, and to describe the changes that the tokens (the ``state variables'') are subject to when passing through these layers as a dynamical system \cite{dutta2021redesigning,sander2022sinkformers,lu2019understanding,geshkovski2023emergence,geshkovski2025mathematical,wu2024role}. 
Even though a natural setting for such model is discrete-time and difference equations, most analyses carried out so far have opted for a continuous-time setting with ordinary differential equations \cite{abellaconsensus,altafini2026multistability,geshkovski2023emergence,geshkovski2025mathematical}, in which the finite number of layers is replaced by a continuum of layers. 
The main reason is that most properties that one can investigate with a dynamical system are asymptotic properties, like for instance the existence and stability of the equilibria, which anyway requires the time axis (ie., the number of layers) to go to $ \infty$.
Another reason is that continuous-time dynamics are more amenable to be investigated using e.g. mean-field methods, quite popular in this context  \cite{bruno2025emergence,burger2025analysis,geshkovski2024dynamic}.

Apart from continuous-time vs discrete-time, in the literature the setting of choice for a mathematical analysis of self-attention dynamics entails several simplifications: the initial embedding and positional encoding are disregarded; a single head is normally considered instead of a multihead; all layers are assumed to have the same weights (weight-sharing or ``universal'' transformers \cite{dehghani2018universal}, which leads to a time-invariant dynamical system); the attention matrix is full (each token attend to all other tokens); the layernorm is replaced by an $ L_2 $ norm (which renders the tokens of unit length, that is, makes them evolve on the unit sphere); the weight matrices are restricted to special structure (sometimes equal to the identity \cite{geshkovski2025mathematical}, other times symmetric \cite{altafini2026multistability}); the FFN part is ignored. 

For such simplified continuous-time model, it has been shown in \cite{geshkovski2023emergence} that the tokens tend to cluster into identical vectors.
In particular, \cite{abella2024asymptotic,abellaconsensus} rephrased this feature as a consensus-like property, a familiar framework in the multiagent system literature \cite{olfati2007consensus}. 
Experimentally, this phenomenon is well-known under the names of  oversmoothing, rank collapse, or token-uniformity \cite{dong2021attention,dovonon2024setting,noci2022signal,nguyen2023mitigating,scholkemper2024residual,shi2022revisiting,zhai2023stabilizing}.
Further analysis carried out in \cite{altafini2026multistability} provides a complete classification of the possible equilibria, and an analysis of their stability and convergence properties, showing that the eigenvectors of the value matrix play a prominent role: most token trajectories align in fact with the eigenvector associated to the principal eigenvalue of the value matrix, at least in the cases in which the  value matrix is symmetric and with a positive strictly dominant eigenvalue.

While it is mentioned in several papers that the discrete-time model of a self-attention dynamics behaves like its continuous-time counterpart, an explicit, thorough stability analysis of the discrete-time case is still missing (an exception is \cite{abella2024asymptotic} where discrete-time consensus is proven in some special cases).
Such an analysis is the main contribution of this paper. 
What we obtain is that indeed the behavior of the discrete-time self-attention model reflects that of its continuous-time counterpart for what concerns the classification of the equilibria and their stability properties. 
Consensus (when all token vectors become identical), bipartite consensus (when some tokens converge to a point on the unit sphere and the others converge to the antipodal point), polygonal equilibria, and clustering (non-antipodal) equilibria are the four classes of possible equilibria. 
The first two classes can occur only at the eigenvectors of the value matrix, and in particular consensus can be asymptotically stable only when it is aligned with the principal eigenvector of the value matrix, while all other consensus points can be shown to be unstable. 
Bipartite consensus equilibria can be (sometimes) locally asymptotically stable, and coexist with the ``main'' consensus attractor, leading to a multistable system as in the continuous-time case \cite{altafini2026multistability}. 
Proving these asymptotic properties requires to compute explicitly the Jacobian of the system and its eigenvalues/eigenvectors, which is significantly more challenging than in the continuous-time case. 
These explicit calculations, which constitute the main technical content of the paper, yield several nontrivial results, most notably the fact that convergence to consensus is always associated to the principal (most positive) eigenvalue of the value matrix, rather than to the eigenvalue of largest absolute value (i.e., equal to the spectral radius).

As a further extension w.r.t. the existing literature, we can provide a stability analysis also for the multihead case. In particular, we can still show asymptotic stability of  the consensus equilibrium at the principal eigenvector of a matrix which is the multihead equivalent of the value matrix (we call this the ``layer matrix'', and it is the product of a sum of suitably zero-padded value matrices at the various heads with the output matrix), meaning that the dynamical behavior of the multihead system is similar to that of the single head case. 

An advantage of considering a discrete-time self-attention dynamics is that the learned weight matrices of a real transformer can be used to test the model.
While in most transformers these weight matrices change from layer to layer (leading to a time-varying dynamics), the special class of universal transformers \cite{dehghani2018universal} has the property of sharing weights across layers (leading to the time-invariant dynamics investigated in this paper). 
A simple example of such universal architecture is ALBERT  (``A Lite BERT'')  \cite{lan2019albert}. 
We show in the paper that under (some of) the simplifying assumptions mentioned above ($ L_2 $ regularization, no FFN), indeed the discrete-time self-attention dynamics obtained by choosing the ALBERT weight matrices converges to a consensus point aligned with the principal eigenvector of our layer matrix.
This confirms that the theoretical analysis is sound, and can provide useful predictions also on the finite-time horizon given by the number of layers of a real transformer.

\section{Model formulation}
\label{sec:single-head-analysis}
We consider a simplified model of a transformer consisting of $T$ layers, each characterized by a single head and by shared weight matrices (universal transformer), and in which we disregard the FFN part.  A residual connection is included in the model, while as normalization we use $ L_2 $ normalization, which amounts to considering unit vectors, i.e., vectors that evolve on the unit sphere. 
To uniform ourselves with the standard dynamical systems notation, we treat tokens (i.e., state vectors) as  columns vectors, see also \cite{abella2024asymptotic,geshkovski2025mathematical}. 
Consider $n$ tokens represented as unit length column vectors $ \bx_i \in \mathbb{S}^{d-1} \subset \mathbb{R}^d$, $ i=1, \ldots, n $. 
Denote $ \bQ,\bK\in \mathbb{R}^{m\times d}$ the query and key matrices (for simplicity and without loss of generality we hereafter assume $ m=d$)  and $ \bV \in \mathbb{R}^{d \times d }$ the value matrix. 
Because of our column-token representation, these matrices correspond to the transposes of the weight matrices normally considered in the machine learning  literature \cite{vaswani2017attention, Likhosherstov2021OnTE, chowdhury2022learningtransformerkernel, ustaomeroglu2025theoreticalstudyhyperselfattention} (more on this in Section~\ref{sec:ext-multi-head}). 
Since we are dealing with a single head, a fourth matrix (denoted output matrix) is not required, and is therefore for now omitted.

The following discrete-time difference equation encodes the single-head self-attention dynamics we consider, including a residual connection and an $ L_2 $ layer normalization:
\begin{equation}
\label{eq:update}
\bx_i^+ = \frac{\bx_i +  \eta \bV \sum_{j=1}^n A_{ij}(\bx)  \bx_j }{\| \bx_i + \eta \bV\sum_{j=1}^n A_{ij}(\bx)  \bx_j \|},\quad i=1, \ldots, n
\end{equation}
where $ A_{ij}(\bx) = \frac{e^{\langle \bQ \bx_i, \bK \bx_j \rangle } }{\sum_{\ell=1}^n e^{\langle \bQ \bx_i, \bK \bx_\ell \rangle } } $ is the attention that the token $ \bx_i $ gives to the token $ \bx_j$, computed through a softmax function ($ \bx $ is the stack of $ \bx_1, \ldots, \bx_n $ vectors) and 
where we have added an extra parameter $ \eta>0 $ (``step size''), which balances the importance of the self-attention part against the residual connection \cite{noci2022signal}.  
The (row stochastic) attention matrix $ A(\bx) $ is assumed full, i.e., each token attends to every other token.
In \eqref{eq:update} we are implicitly interpreting the tokens as layer-indexed state variables: $ \bx_i = \bx_i(t) $ and $ \bx_i^+ = \bx_i(t+1)$, for some layer index $ t = 1, \ldots, T$.
We let $\|\cdot \|$ denote the Euclidean vector norm and the operator norm for matrices.
For the $i$-th token, define the tangent projector on $ \mathbb{S}^{d-1}$ at $\bx_i$ by $ \Pi_{\bx_i} := I -\bx_i  \bx_i ^T$, which satisfies $\Pi_{\bx_i} \bx_i =0$ and $\Pi_{\bx_i} ^T=\Pi_{\bx_i} $.

We are interested in investigating the asymptotic behavior of \eqref{eq:update} when the number of layers $T$ goes to $ \infty$, and in particular in computing the equilibria of \eqref{eq:update} and studying their stability properties. 
For that, it is convenient to investigate first a simpler system, which, for analogy with the continuous-time case, we call the multiagent discrete-time Oja system \cite{altafini2026multistability}.
This system has the same structure as \eqref{eq:update} but lacks the attention matrix.

\section{Multiagent discrete-time Oja system}
\label{sec:multi-oja}
When we consider uniform attention \cite{noci2022signal}, i.e., we take $ A(\bf{x})=\frac{1}{n} \mathds{1} \mathds{1}^T$, then instead of \eqref{eq:update} we obtain the simpler system
\begin{equation}
\label{eq:update-moja}
\bx_i^+ = \frac{\bx_i +  \frac{\eta}{n} \bV \sum_{j=1}^n   \bx_j }{\| \bx_i + \frac{\eta}{n} \bV\sum_{j=1}^n \bx_j \|}.
\end{equation}
Let $ \by= \bV \bm{m} ( \bx)  =  \frac{1}{n}\bV \sum_{j=1}^n  \bx_j \in \mathbb{R}^d$ be the total influence of all agents on agent $i$ (which is the same for all agents).
Splitting \eqref{eq:update-moja} into parts before/after the normalization we can also write:
\beq
\tilde{\bx}_i = \bx_i + \eta \by 
\quad \text{and} \quad \bx_i^+ = \frac{\tilde{\bx}_i}{\| \tilde{\bx}_i \|} .
\label{eq:update-moja2}
\eeq

We make the following assumption which holds throughout the rest of this and next section.

\begin{assumption} 
\label{ass:V-sym}
The value matrix $ \bV\in \mathbb{R}^{d\times d } $ is nonsingular and symmetric, of eigenvalues $ \lambda_1>  \lambda_2 \geq \ldots \geq \lambda_d $ with $ \lambda_1 >0 $ simple and positive.
\end{assumption}
Let $ \bv_1, \ldots, \bv_d $ be the associated eigenvectors, normalized s.t. $ \| \bv_k \|=1$.

\subsection{Equilibria}

Let us first give our definition of consensus and bipartite consensus equilibria, which we request to be collinear with the eigenvectors of $\bV$ (the reason will become clear later on when we study the full model \eqref{eq:update}). 
The points $ \bx_1, \ldots, \bx_n \in \mathbb{S}^{d-1} $ are said to be in a {\em consensus equilibrium} if $ \bx_i = \bx_j = \bv_k $ $ \forall \, i, j = 1, \ldots, n$ (or $ \bx_i = \bx_j = - \bv_k $ $ \forall \, i, j = 1, \ldots, n$) and for some $ k=1, \ldots, d$. 
They are said to be in a {\em bipartite consensus equilibrium} if $ \bx_i =  \bv_k  $ for $ i \in  \mathcal{V}_1 \subset \mathcal{V} $  and $ \bx_i =  - \bv_k  $ for $ i \in  \mathcal{V} \setminus \mathcal{V}_1  $, for some $ k=1, \ldots, d$.

\begin{lemma}
\label{lem:moja-equil1}
Let $ \bx = [ (\bx_1)^T \, \ldots  (\bx_1)^T ]^T$ be an equilibrium point of the system \eqref{eq:update-moja}. 
Then $ \bx $  belongs to one of the following classes:
\benu
\item consensus: $ \bx_i = \bv_k $ $ \forall \, i$, and $ k=1, \ldots, d$;
\item bipartite consensus: $ \bx_i = \pm \bv_k $ $ \forall \, i$ and $ k=1, \ldots, d$;
\item polygonal equilibria: $ \{ {\bx_i \in \mathbb{S}^{d-1}} \; \text{s.t.} \;  \bV \sum_{j=1}^n  \bx_j =0\}. $
\eenu
\end{lemma}

\proof
An equilibrium point $ \bx_i $ satisfies $ \bx_i^+ = \bx_i$ for all $i$, hence, from \eqref{eq:update-moja2}, $ \bx_i = \frac{\tilde{\bx}_i}{\| \tilde{\bx}_i \|}  $, which implies that $ \tilde{\bx}_i$ and $ \bx_i $ must be collinear, i.e., $ \tilde{\bx}_i = \varphi  \bx_i $ for some $ \varphi \in \mathbb{R}$.
If $ \by =   \frac{1}{n}\bV \sum_{j=1}^n  \bx_j$, then from \eqref{eq:update-moja2}, $ \tilde{\bx}_i = \bx_i + \eta \by  $ holds iff also $\by  $ and $ \bx_i $ are collinear or $ \by =0$.
This condition can be written equivalently using projections as 
 \[
\Pi_{\bx_i} \by =  (I - \bx_i \bx_i^T )\by   =0 .
 \]
 There are 3 cases in which this  can happen:
\benu
\item consensus: $ \by = \varphi_i \bx_i $, for some scalar $ \varphi_i >0$;
\item bipartite consensus: $ \by = -\varphi_i \bx_i $, for some scalar $\varphi_i >0$;
\item polygonal equilibria: $ \by =0$.
\eenu
In fact, in the first two cases, $ \Pi_{\bx_i} \bx_i = \pm  \varphi_i (\bx_i - \bx_i  \bx_i^T \bx_i ) = 0$, since $ \bx_i^T \bx_i =1$. The third case follows trivially.
To show that consensus must be an eigenvector of $\bV$, observe that at this equilibrium point we have  $ \by= \bV \bx_i $, since $ \bx_i= \bx_j $.
Combining with $ \by= \varphi_i \bx_i $, we have $  \bV \bx_i =\varphi_i \bx_i $, i.e., $ \bx_i $ is an eigenvector of $\bV$ and $ \varphi_i$ one of its eigenvalues. The argument for bipartite consensus is identical. 

\qed

\begin{remark}
For a given $ \bv_k$ there are $ 2^n $ possible consensus or bipartite consensus equilibria.
These are always paired by a global symmetry w.r.t. the origin, i.e., if $ \bx_i = \bv_k $ $ \forall\, i$  is a consensus equilibrium point, its antipodal point $ \bx_i = - \bv_k $ $ \forall\, i$ is also a consensus equilibrium point, and similarly for the bipartite consensus equilibria.
This is implicitly assumed from now on, but  not explicitly written, to avoid overloading the formulation of the results.
\end{remark}

\subsection{Stability}
The stability properties of the multiagent Oja system \eqref{eq:update-moja} are summarized in the following theorem. 

\begin{theorem}
\label{thm:moja-DT}
For the system \eqref{eq:update-moja}, under Assumption~\ref{ass:V-sym}, if 
$ \eta< \min \left( \frac{2}{ |\lambda_1 + \lambda_d |}, \, \frac{1}{2 \max \{  \lambda_1 , | \lambda_d | \} }  \right) $, 
the principal consensus equilibrium $ \bx_i=\bv_1 $ $\forall \, i$ is asymptotically stable, while all other consensus equilibria, all bipartite consensus equilibria (including those aligned with $ \bv_1$), and all polygonal equilibria are unstable.
The trajectories of all tokens $ \bx_i(t)$ converge to $ \bv_1 $ (or to $ - \bv_1 $) for almost all initial conditions $ \bx_i(0)$. 
\end{theorem}

The theorem is proven through a series of lemmas. 
Let $ \delta_{ih} $ be a delta function ($ \delta_{ih}=1 $ if $ i=h$ and $ \delta_{ih}=0$ otherwise). 
Let  the right-hand side of \eqref{eq:update-moja}  be $ f_i(\bx)  =  \frac{\bx_i +  \frac{\eta}{n} \bV \sum_{j=1}^n   \bx_j }{\| \bx_i + \frac{\eta}{n} \bV\sum_{j=1}^n \bx_j \|}$, and denote $ f(\bx) =\begin{bmatrix} f_1(\bx)^T & \ldots & f_n(\bx)^T\end{bmatrix}^T $ the $ (nd) $-dimensional vector field associated to the stacked state vector $ \bx= \begin{bmatrix} \bx_1^T & \ldots & \bx_n ^T \end{bmatrix}^T$. 
The Jacobian $ J(\bx) = \pde{f(\bx)}{\bx} = \begin{bmatrix}\pde{f_i(\bx)}{\bx_h} \end{bmatrix} $ is composed of a diagonal part and a part that repeats itself on all $ (i, h)$ blocks.

\begin{lemma}
\label{lem:Jac-moja}
The Jacobian $ J(\bx) $ of the system \eqref{eq:update-moja} has the following blocks:
\beq
\begin{split}
\pde{f_i(\bx)}{\bx_h}  =& \frac{\frac{\eta}{n} \bV}{\| \bx_i +  \frac{\eta}{n} \bV \sum_j \bx_j \| } 
-  ( \bx_i + \frac{\eta}{n} \bV \sum_j \bx_j ) \cdot \\
& \cdot \frac{ \frac{\eta}{n} \left( \bx_i^T \bV + \frac{\eta}{n} \left( \sum_j \bx_j^T \bV^T (I - \bx_i \bx_i^T )\bV \right) \right) }
{\| \bx_i +  \frac{\eta}{n} \bV \sum_j \bx_j \|^2 } \\
& + \left( \frac{I}{\| \bx_i +  \frac{\eta}{n} \bV \sum_j \bx_j \| }  
-  ( \bx_i + \frac{\eta}{n} \bV \sum_j \bx_j ) \cdot \right. \\
&\left.  \cdot \frac{ \frac{\eta}{n} \left( \sum_j \bx_j^T \bV^T (I -\frac{\eta}{n} \bx_i \sum_\ell \bx_\ell^T \bV ) \right)}
{\| \bx_i +  \frac{\eta}{n} \bV \sum_j \bx_j \|^2 }\right) \delta_{ih}
\end{split}
\label{eq:Jacob-moja}
\eeq
\end{lemma}

\proof
Let us first express the series expansion of the normed term:
\[
\begin{split}
& \| \bx_i +  \frac{\eta}{n} \bV \sum_j \bx_j \| =  \left(  (  \bx_i +  \eta \by )^T (  \bx_i + \eta \by) \right)^{\frac{1}{2}}  \\
& = \left( 1 + 2  \eta \bx_i ^T \by + \eta^2 \| \by \|^2   \right)^{\frac{1}{2}} \\
& \approx  1 + \frac{1}{2} \left( 2  \eta \bx_i ^T \by + \eta^2 \| \by \|^2  \right) - \frac{1}{2}  \eta^2 ( \bx_i^T \by ) ^2 + O(\eta^3)  \\
& =  1 + \frac{\eta}{n} \bx_i^T \bV \sum_j \bx_j \\
& + \frac{\eta^2}{2 n^2} \left( \sum_j \bx_j^T \bV^T \bV \sum_j \bx_j - ( \bx_i^T \bV \sum_j \bx_j)^2 \right) \! \!  + O(\eta^3) 
\end{split}
\]
where we have used the binomial series expansion ($ (1 + \xi)^{\frac{1}{2}}\approx 1 + \frac{\xi}{2} - \frac{ \xi^2}{8}  + O(\xi^3)  $), disregarding terms of order higher than 2 in $ \eta$.
Differentiating this term
\[
\begin{split}
& \pde{\| \bx_i +  \frac{\eta}{n} \bV \sum_j \bx_j \|}{\bx_h } = \\
&  =  \frac{\eta}{n} \bx_i^T \bV  + \frac{\eta^2}{n^2}  \sum_j \bx_j^T \bV^T ( I - \bx_i \bx_i^T ) \bV   \\
& + \left( \frac{\eta}{n} \sum_j \bx_j^T \bV^T - \frac{\eta^2}{n^2}  \sum_j \bx_j^T \bV^T \bx_i \sum_\ell \bx_\ell^T \bV \right) \delta_{ih}
\end{split}
\]
Inserting this expression into
\[
\begin{split}
\pde{f_i(\bx)}{\bx_h}  =& \frac{(I \delta_{ih} + \frac{\eta}{n} \bV)  \| \bx_i +  \frac{\eta}{n} \bV \sum_j \bx_j \|}
{\| \bx_i +  \frac{\eta}{n} \bV \sum_j \bx_j \|^2} \\
&  - \frac{ (  \bx_i +  \frac{\eta}{n} \bV \sum_j \bx_j ) \pde{\| \bx_i +  \frac{\eta}{n} \bV \sum_j \bx_j \|}{\bx_h }  }
{\| \bx_i +  \frac{\eta}{n} \bV \sum_j \bx_j \|^2}
\end{split}
\]
the result follows.
\qed

\begin{lemma}
\label{lem:moja-F-eigenvalues}
For the system \eqref{eq:update-moja}, under Assumption~\ref{ass:V-sym}, at a consensus equilibrium $ \bx_i=\bv_k $ $\forall \, i$ the Jacobian $ J(\bx) $ has the following 3 classes of eigenvalues:
\benu
\item if $ 1+ \eta \lambda_k >0$, $ \frac{1-\eta \lambda_k + \eta^2 \lambda_k^2 }{1+ \eta \lambda_k } $ of multiplicity $ n$, while if $ 1+ \eta \lambda_k <0 $, $ \frac{-1-3 \eta \lambda_k + \eta^2 \lambda_k^2 }{1+ \eta \lambda_k } $ of multiplicity $ 1$ and $ \frac{-1-\eta \lambda_k + \eta^2 \lambda_k^2 }{1+ \eta \lambda_k } $ of multiplicity $ n-1$.
\item $ \frac{1+ \eta \lambda_h }{| 1+ \eta \lambda_k |} $ of multiplicity 1, $ h=1, \ldots, k-1, k+1, \ldots, d$. 
\item $  \frac{1 }{| 1+ \eta \lambda_k |} $ of multiplicity $ nd-d-n+1$.
\eenu
\end{lemma}

\proof
When computing the Jacobian \eqref{eq:Jacob-moja} at the consensus $ \bx_i=\bv_k $ $\forall \, i$, we get
\beq
 \begin{split}
 J(\bv_k) & = \left. \pde{f(\bx)}{\bx} \right|_{\bx_i = \bv_k} \\
 & = \mathds{1} \mathds{1}^T \otimes \left(  \frac{\sign(1+ \eta \lambda_k) \frac{\eta}{n} \bV - \frac{\eta}{n} \lambda_k \bv_k \bv_k^T }{1 + \eta \lambda_k }  \right) \\
 & \;\;  + I \otimes \left(   \frac{\sign(1+ \eta \lambda_k)  I - \eta \lambda_k (1-\eta \lambda_k ) \bv_k \bv_k^T }{1 + \eta \lambda_k } \right).
 \end{split}
 \label{eq:JacobF-moja}
 \eeq
 where we have used $ \bv_k^T ( I - \bv_k \bv_k^T) =0$.
 The first term represents a factor present in all entries of $J(\bv_k) $, while the second one is present only on the diagonal.
 For $ J(\bv_k) $ there are 3 classes of eigenvectors:
 \benu
 \item The eigenvectors of the first class depend on the sign of $ 1+ \eta \lambda_k $. If $ 1+ \eta \lambda_k >0$, then choose $ \bp = [ 0 \ldots 0 \; \bv_k \; 0 \ldots 0 ] $ and compute the two components of $ J(\bv_k) $:
 \[
 \begin{cases}
&  \frac{\sign(1+ \eta \lambda_k) \frac{\eta}{n} \bV - \frac{\eta}{n} \lambda_k \bv_k \bv_k^T }{1 + \eta \lambda_k } \bv_k 
= \frac{\frac{\eta}{n} (1 -1)}{1+ \eta \lambda_k } \bv_k  =0 \\ 
& 
 \frac{\sign(1+ \eta \lambda_k)  I - \eta \lambda_k (1-\eta \lambda_k ) \bv_k \bv_k^T }{1 + \eta \lambda_k } \bv_k =
  \frac{ 1 - \eta \lambda_k ( 1- \eta \lambda_k ) }{1+ \eta \lambda_k } \bv_k
  \end{cases}
\]
meaning that 
\[
J(\bv_k ) \bp = \frac{ 1 - \eta \lambda_k  + \eta^2 \lambda_k^2  }{1+ \eta \lambda_k } \bp
\]
i.e., $ \frac{ 1 - \eta \lambda_k  + \eta^2 \lambda_k^2  }{1+ \eta \lambda_k }  $ is an eigenvalue of $ J(\bv_k)$ of multiplicity $ n$ (as there are $n$ possible slots for $ \bv_k $ in $ \bp$).
If instead $ 1+ \eta \lambda_k <0 $, then choose the following $ n$ eigenvectors
\[
\bp = \begin{cases} \bp_1 =  [\bv_k \; \bv_k \ldots \bv_k ]  \\
 \bp_\ell = [\phi_1^\ell \bv_k \;  \phi_2^\ell \bv_k \ldots \phi_n^\ell \bv_k ] \; \text{s.t.} \; \sum_i \phi_i^\ell =0 , \\
 \qquad \qquad \qquad  \qquad \qquad \qquad  \;\; \ell=2, \ldots, n
 \end{cases}
 \]
 In correspondence of $ \bp_1 $, one gets 
 \[
 \begin{cases}
&  \frac{\sign(1+ \eta \lambda_k) \frac{\eta}{n} \bV - \frac{\eta}{n} \lambda_k \bv_k \bv_k^T }{1 + \eta \lambda_k } \bv_k 
=\frac{\eta}{n}  \frac{ - 2 \lambda_k }{1+ \eta \lambda_k } \bv_k   \\ 
& 
 \frac{\sign(1+ \eta \lambda_k)  I - \eta \lambda_k (1-\eta \lambda_k ) \bv_k \bv_k^T }{1 + \eta \lambda_k } \bv_k =
  \frac{ -1 - \eta \lambda_k + \eta^2 \lambda_k^2 }{1+ \eta \lambda_k } \bv_k
  \end{cases}
\]
hence 
\[
\begin{split}
J(\bv_k) \bp_1 & = \eta  \frac{ - 2 \lambda_k }{1+ \eta \lambda_k } \bp_1 +  \frac{ -1 - \eta \lambda_k + \eta^2 \lambda_k^2  }{1+ \eta \lambda_k } \bp_1 \\
& = \frac{-1-3 \eta \lambda_k + \eta^2 \lambda_k^2 }{1+ \eta \lambda_k }  \bp_1 
\end{split}
\]
while in correspondence of $ \bp_\ell $ one gets
\[
\begin{split}
J(\bv_k) \bp_\ell & = \begin{bmatrix} \sum_i \phi_i^\ell \left( \frac{\eta}{n}   \frac{ - 2 \lambda_k }{1+ \eta \lambda_k } \right) \bv_k +  \frac{ -1 - \eta \lambda_k + \eta^2 \lambda_k^2  }{1+ \eta \lambda_k } \phi_1 \bv_k \\
\vdots \\
\sum_i \phi_i^\ell \left( \frac{\eta}{n}   \frac{ - 2 \lambda_k }{1+ \eta \lambda_k } \right) \bv_k +  \frac{ -1 - \eta \lambda_k + \eta^2 \lambda_k^2  }{1+ \eta \lambda_k } \phi_n \bv_k
\end{bmatrix} \\
& =  \frac{ -1 - \eta \lambda_k + \eta^2 \lambda_k^2  }{1+ \eta \lambda_k } \bp_\ell 
\end{split}
\]
i.e., $  \frac{ -1 - \eta \lambda_k + \eta^2 \lambda_k^2  }{1+ \eta \lambda_k } $ is an eigenvalue of multiplicity $ n-1$. 

\item For the second class of eigenvectors choosing $ \bq_h = [ \bv_h\, \ldots \, \bv_h] $ for $ h\neq k$, leads to 
\[
 \begin{cases}
&  \frac{\sign(1+ \eta \lambda_k) \frac{\eta}{n} \bV - \frac{\eta}{n} \lambda_k \bv_k \bv_k^T }{1 + \eta \lambda_k } \bv_h 
=\frac{\eta}{n}  \frac{  \lambda_h }{|1+ \eta \lambda_k |} \bv_h  \\ 
& 
 \frac{\sign(1+ \eta \lambda_k)  I - \eta \lambda_k (1-\eta \lambda_k ) \bv_k \bv_k^T }{1 + \eta \lambda_k } \bv_h =
  \frac{ 1  }{|1+ \eta \lambda_k |} \bv_h
  \end{cases}
\]
and hence to 
\[
J(\bv_k) \bq_h = \frac{1 + \eta \lambda_h }{|1+ \eta \lambda_k |} \bq_h 
\]
i.e., $ \frac{1 + \eta \lambda_h }{|1+ \eta \lambda_k |} $ is an eigenvalue of $ J(\bv_k)$. There are $ d-1 $ such eigenvalues. 
\item The remaining $ nd-d-n+1 $ eigenvectors are assembled by considering vectors $ \br =\begin{bmatrix} (\bz^1)^T & (\bz^2)^T & \ldots  & (\bz^n)^T \end{bmatrix}^T$ s.t. $ \bv_k^T \bz^i =0 $ for all $ i =1, \ldots, n $ and all $ h=1, \ldots, k-1, k+1, \ldots, d $ and $ \bq_h^T \br =0 $. Since the number of such constraints is $ d-1+n$, there exist $ nd- n-d+1 $ such vectors $ \bm{r}$, and they can always be chosen so that  $ \sum_{i=1}^n  \bz^i =0 $.
Computing,
\[
 \begin{cases}
&  \frac{\sign(1+ \eta \lambda_k) \frac{\eta}{n} \bV - \frac{\eta}{n} \lambda_k \bv_k \bv_k^T }{1 + \eta \lambda_k } \bz^i 
= \frac{\eta}{n} \frac{ \bV }{|1+ \eta \lambda_k |} \bz^i  \\ 
& 
 \frac{\sign(1+ \eta \lambda_k)  I - \eta \lambda_k (1-\eta \lambda_k ) \bv_k \bv_k^T }{1 + \eta \lambda_k } \bz^i =
  \frac{ 1  }{|1+ \eta \lambda_k |} \bz^i
  \end{cases}
\]
hence
\[
J(\bv_k) \br =   \frac{ 1  }{|1+ \eta \lambda_k |} \begin{bmatrix}  \frac{\eta}{n} \bV \sum_i \bz_i + \bz_1 \\
\vdots \\
 \frac{\eta}{n} \bV \sum_i \bz_i + \bz_n
\end{bmatrix} =   \frac{ 1  }{|1+ \eta \lambda_k |} \br
\]
\end{enumerate}
\qed

\begin{remark}
\label{rem:two-classes}
The discrete-time system \eqref{eq:update} can be seen as a map from the manifold $ (\mathbb{S}^{d-1})^n $ to itself. 
Hence its tangent space is $ T_{\bx}  (\mathbb{S}^{d-1})^n $, of dimension $ n (d-1)$, while the Jacobian $ J(\bx)$ computed above is in dimension $ nd $. 
Such Jacobian is what is sometimes called the ``extrinsic Jacobian'' living in the ambient manifold $ \mathbb{R}^{ nd \times nd}$ rather than in $ T_{\bx}  (\mathbb{S}^{d-1})^n $. 
What decides stability is however the ``intrinsic Jacobian'' living on  $ T_{\bx}  (\mathbb{S}^{d-1})^n $. For Lemma~\ref{lem:moja-F-eigenvalues}, if we look at the eigenvectors, we can easily deduce that the first class of eigenvectors (of cardinality $n$) does not belong to $ T_{\bx}  (\mathbb{S}^{d-1})^n $, while the other two classes are always in $ T_{\bx}  (\mathbb{S}^{d-1})^n $, hence it is only these last two classes that decide the stability character of the equilibria.
\end{remark}

\begin{lemma}
\label{lem:moja-consensus-stab}
For the system \eqref{eq:update-moja}, under Assumption~\ref{ass:V-sym}, and, if $ - \lambda_d > \lambda_1 $,  for $ \eta< \frac{2}{ |\lambda_1 + \lambda_d |} $, then the principal consensus equilibrium point, $ \bx_i=\bv_1 $ $\forall \, i$, is asymptotically stable.
For $ k=2, \ldots, d $, the consensus equilibria $ \bx_i=\bv_k $ $\forall \, i$ are all unstable. 
\end{lemma}

\proof
Consider the  principal consensus equilibrium point $ \bv_1$ and the eigenvalues of $ J(\bv_1)$ computed in Lemma~\ref{lem:moja-F-eigenvalues}. 
Notice from Remark~\ref{rem:two-classes} that only the second and third class of eigenvalues matter for the stability.
Under Assumption~\ref{ass:V-sym}, $ 1 + \eta \lambda_1 >1$, hence the 3rd class of eigenvalues is inside the unit disk. When $ | \lambda_d | < \lambda_1 $, then also the second class is inside the disk for all $ h=2, \ldots, d$, regardless of the sign of $ \lambda_h$, while when $ - \lambda_d >\lambda_1 $ we need to have $  \frac{1+ \eta \lambda_d }{ 1+ \eta \lambda_1 } >-1 $, i.e., $ \eta <  \frac{-2}{\lambda_1 + \lambda_d } = \frac{2}{| \lambda_1 + \lambda_d |} $. 

When $ k>1$ (consensus at non-principal eigenvectors), then one of the eigenvalues of $ J(\bv_k) $ is $ \frac{1 + \eta \lambda_1 }{|1+ \eta \lambda_k |}  $ which is always $ >1 $. This is obvious when $ \lambda_k \geq 0 $, as $ \lambda_k< \lambda_1$. What needs to be checked is the case of $ \lambda_k $ negative with $ - \lambda_k \gg \lambda_1 $. 

For this case we consider the two situations
\benu
\item If $ 1 + \eta \lambda_k >0 $, then we have $ \frac{1 + \eta \lambda_1 }{ 1+ \eta \lambda_k }>1$ because this is equivalent to $ 1 + \eta \lambda_1 > 1+ \eta \lambda_k $. 

\item If $ 1 + \eta \lambda_k <0 $, then we must show that it is $ - \frac{1 + \eta \lambda_1 }{ 1+ \eta \lambda_k }>1$, i.e., that $ 1 + \eta \lambda_1 > - 1 - \eta \lambda_k$, or $ \eta (\lambda_1 + \lambda_k) >-2$. As $ \lambda_1 + \lambda_k<0$, the expression holds if it holds for the largest possible value of $ \eta $.
Taking $ \eta = \frac{2}{| \lambda_1 + \lambda_d|}$, we obtain the following inequality $ - \frac{2(\lambda_1 + \lambda_k)}{\lambda_1 + \lambda_d}>-2 $, or, equivalently, $ \lambda_1 + \lambda_k > \lambda_1 + \lambda_d$ which is always true. 
\eenu
In both situations $ J(\bv_k) $ has an eigenvalue outside the unit disk. 
\qed

At a bipartite consensus point $ \bx_i = \pm \bv_k $, split the $n$ tokens into two sets $ \mathcal{V}_1 $ and $ \mathcal{V}_2 $, $ \mathcal{V}_1\cup \mathcal{V}_2 = \{ 1, \ldots, n \}$,  according to whether $ \bx_i =\bv_k $ or $ \bx_i =- \bv_k$ at equilibrium. Assume that $ n_1 = | \mathcal{V}_1|  $ tokens are equal to $ \bv_k $ and $ n_2 = | \mathcal{V}_2 | $ equal to $ -\bv_k$, with $ n_1 + n_2 =n$. 
Let $ \nu=\frac{n_1-n_2}{n}$, where we consider w.l.o.g. $ n_1 \geq n_2 $ (hence $ \nu\geq 0$). 

%
 
\begin{lemma}
\label{lem:moja-bip-consensus-stab}
For the system \eqref{eq:update-moja}, under Assumption~\ref{ass:V-sym}, the bipartite consensus equilibria, $ \bx_i=\pm \bv_k $ $\forall \, i$ are all unstable $ \forall \, k =1, \ldots, d$ and $ \forall \, \eta>0$.
\end{lemma}

\proof
At a bipartite consensus equilibrium there exists at least one index $ i \in \mathcal{V}_1 $. Consider the perturbation $ \bx_i = \bv_k + \bw_i  $ with $ \bv_k^T \bw_i =0$ $ \forall \, i$. For this perturbation, the unnormalized next iterate is 
\[
\begin{split}
\tilde{\bx}_i & = \bv_k+\bw_i + \frac{\eta}{n} \bV\sum_j (\bv_k + \bw_j) \\
& = ( 1+ \eta \nu \lambda_k ) \bv_k + \bw_i + \frac{\eta}{n} \bV \sum_j \bw_j 
\end{split}
\]
and its norm
\[
\| \tilde{\bx}_i \| = \left( \alpha^2\left( 1+ \frac{\| \bu_i \|^2 }{\alpha^2} \right) \right)^{\frac{1}{2}}
\]
where $ \alpha := 1 + \eta \nu \lambda_k $, $ \bu_i := \bw_i + \frac{\eta}{n} \bV \sum_j \bw_j $, and we have used $ \bv_k^T \bv_k =1 $ and $ \bv_k^T \bw_i =0 $.
Hence, using the binomial series expansion ($ (1 + \xi)^{-\frac{1}{2}}\approx 1 - \frac{\xi}{2}  + h.o.t. $), 
\[
\begin{split}
\bx_i^+ &= \frac{\tilde{\bx}_i}{\| \tilde{\bx}_i \|}
= \frac{ \alpha \bv_k + \bu_i }{  \alpha \left( 1+ \frac{\| \bu_i \|^2 }{\alpha^2} \right) ^{\frac{1}{2}} } 
\approx \bv_k + \frac{\bu_i}{\alpha} + h.o.t. \\
&= \bv_k + \frac{ \bw_i + \frac{\eta}{n} \bV \sum_j \bw_j }{ 1+ \eta \nu \lambda_k }  + h.o.t.\\
& \approx  \bv_k + ( 1- \eta \nu \lambda_k ) ( \bw_i + \frac{\eta}{n} \bV \sum_j \bw_j ) 
\end{split}
\]
where we have used the expansion $ \frac{1}{1+ \eta \nu \lambda_k } \approx 1-  \eta \nu \lambda_k$. 
The linearized equation is then 
\[
\bw_i^+ =  ( 1- \eta \nu \lambda_k )  \bw_i + \frac{\eta}{n} \bV \sum_j \bw_j .
\]
Expanding all $ \bw_i $ in the orthonormal basis $ \bv_1, \ldots, \bv_d $, and denoting $ \bw_i = \sum_\ell \xi^i_\ell \bv_\ell $, 
\[
\bw_i^+ = \sum_\ell (\xi^i_\ell)^+ \bv_\ell =  ( 1- \eta \nu \lambda_k ) \sum_\ell \xi^i_\ell \bv_\ell + \frac{\eta}{n} \sum_{j, \ell} \lambda_\ell \xi^j_\ell  \bv_\ell .
\]
Projecting along $ \bv_1 $, 
\[
(\xi^i_1)^+ =  ( 1- \eta \nu \lambda_k ) \xi^i_1 +\frac{\eta}{n} \lambda_1 \sum_j \xi^j_1 .
\]
If $ \bm{\xi}_1 := [ \xi_1^1 \ldots \xi_1^n ]^T$, then these $ n$ difference equations can be written in vector form as
\[
\bm{\xi}_1^+ =  ( 1- \eta \nu \lambda_k ) \bm{\xi}_1 + \frac{\eta}{n} \lambda_1 \mathds{1} \mathds{1}^T \bm{\xi}_1 .
\]
We need to compute the eigenvalues of the $ n \times n$ matrix $ M(k,1) = ( 1- \eta \nu \lambda_k )I + \frac{\eta}{n} \lambda_1 \mathds{1} \mathds{1}^T $. 
One such eigenvalues is associated to the consensus vector $ \mathds{1} \in \mathbb{R}^d$:
\[
 M(k,1) \mathds{1} = ( 1- \eta \nu \lambda_k ) \mathds{1} +  \eta  \lambda_1 \mathds{1} = ( 1- \eta \nu \lambda_k +\eta \lambda_1 ) \mathds{1} .
 \]
 If $ k\neq 1 $ then either $ \lambda_k >0 $ and $ \lambda_k < \lambda_1$, or $ \lambda_k<0$. Since $ \nu<1$, in both cases $ 1- \eta \nu \lambda_k +\eta \lambda_1 >1 $ and the bipartite equilibrium point is unstable. 
 If $ k=1$, then $ 1+\eta(1-\nu) \lambda_1 >1$, i.e., also in this case we have found an eigenvalue outside the unit disk, and hence the equilibrium point is unstable. 
\qed

\begin{lemma}
\label{lem:moja-polygonal-stab}
For the system \eqref{eq:update-moja}, under Assumption~\ref{ass:V-sym}, all polygonal equilibria are unstable.
\end{lemma}

\proof
At a polygonal equilibrium point $ \bx=\bs =[\bs_1 \, \dots \, \bs_n] $, consider the perturbation $ \bx_i = \bs_i + \bw_i  $ with $ \bs_i^T \bw_i =0$ $ \forall \, i$. For this perturbation, the unnormalized next iterate is 
\[
\begin{split}
\tilde{\bx}_i & = \bs_i+\bw_i + \frac{\eta}{n} \bV\sum_j (\bs_j + \bw_j) \\
& =\bs_i + \frac{\eta}{n} \underbrace{\bV\sum_j \bs_j}_{=0} + \bw_i + \frac{\eta}{n}  \bV \sum_j \bw_j 
\end{split}
\]
and its norm
\[
\| \tilde{\bx}_i \| = \left( 1+ \| \bu_i \|^2 +2  \frac{\eta}{n}  \bs_i^T \bV \sum_j \bw_j  \right)^{\frac{1}{2}}
\]
where $ \bu_i := \bw_i + \frac{\eta}{n} \bV \sum_j \bw_j $, and we have used $ \bs_i^T \bw_i =0 $.
Hence, using the binomial series expansion ($ (1 + \xi)^{-\frac{1}{2}}\approx 1 - \frac{\xi}{2}  + h.o.t. $), 
\[
\begin{split}
\bx_i^+ &= \frac{\tilde{\bx}_i}{\| \tilde{\bx}_i \|}
= \frac{  \bs_i + \bu_i }{ \left( 1+ \| \bu_i \|^2 + 2  \frac{\eta}{n}  \bs_i^T \bV \sum_j \bw_j  \right) ^{\frac{1}{2}} } \\
&\approx (\bs_i + \bu_i) \left(1- \frac{\| \bu_i \|^2}{2} -   \frac{\eta}{n}  \bs_i^T \bV \sum_j \bw_j \right)  + h.o.t. \\
&= \bs_i + \bu_i -   \frac{\eta}{n}  \bs_i^T \bV \sum_j \bw_j \bs_i  + h.o.t.\\
& = \bs_i + \bw_i + \frac{\eta}{n} (I - \bs_i \bs_i^T) \bV \sum_j \bw_j + h.o.t.
\end{split}
\]
The linearized equation is then 
\[
\bw_i^+ = \bw_i + \frac{\eta}{n} (I - \bs_i \bs_i^T) \bV \sum_j \bw_j .
\]
Expanding all $ \bw_i $ and $ \bs_i $ in the orthonormal basis $ \bv_1, \ldots, \bv_d $, and denoting $ \bw_i = \sum_\ell \xi^i_\ell \bv_\ell $ and $ \bs_i = \sum_\ell \zeta^i_\ell \bv_\ell $, 
\[
\begin{split}
\bw_i^+ & = \sum_\ell (\xi^i_\ell)^+ \bv_\ell \\
& =  \sum_\ell \xi^i_\ell \bv_\ell + \frac{\eta}{n} \left( I - \sum_h \zeta^i_h \bv_h \sum_\ell \zeta^i_\ell \bv_\ell^T \right)  \sum_{j, \ell} \lambda_\ell \xi^j_\ell  \bv_\ell \\
& =  \sum_\ell \xi^i_\ell \bv_\ell + \frac{\eta}{n} \left( I - \sum_h \zeta^i_h \bv_h \sum_\ell \zeta^i_\ell  \lambda_\ell \sum_j  \xi^j_\ell  \right).
\end{split}
\]
Recall that $ \bV$ has always at least one positive eigenvalue $\lambda_1 >0$. We can always choose $ \bm{w}$ s.t. each $ \bm{w}_j $ has a nonzero component in the direction of the associated eigenvector $ \bv_1$: $ \xi^j_1\neq 0$.
Projecting along $ \bv_1 $
\[
(\xi^i_1)^+ = \xi^i_1 + \frac{\eta}{n} \left( 1-{\zeta_1^i}  \sum_\ell \zeta^i_\ell  \lambda_\ell \sum_j  \xi^j_\ell \right).
\]
If the perturbation is such that $ \bw_i $ has only a component along $ \bv_1 $, $ \xi_1^i\neq 0$, $ \xi_j^i=0$ $ \forall \, j \neq i$, then
\[
(\xi^i_1)^+ = \xi^i_1 + \frac{\eta}{n} ( 1-{\zeta_1^i}^2) \lambda_1 \sum_j \xi^j_1 .
\]
If $ \bm{\xi}_1 :=  \begin{bmatrix} \xi_1^1 &\ldots &\xi_1^n\end{bmatrix}^T$ and $  \bm{\zeta}_1 = \begin{bmatrix} \zeta^1_1 & \ldots & \zeta^n_1 \end{bmatrix}^T$, then in vector form
\[
\begin{split}
\bm{\xi}_1^+ & =  \bm{\xi}_1 + \frac{\eta}{n} \lambda_1 \left( I- Z_1^2 \right)\mathds{1} \mathds{1}^T \bm{\xi}_1 \\
& = \left( I+ \frac{\eta}{n} \lambda_1 \left( I- Z_1^2 \right)\mathds{1} \mathds{1}^T \right) \bm{\xi}_1 .
\end{split}
\]
where $ Z_1 = \diag(\bm{\zeta}_1)$. 
Since $ \bm{s} $ is not an eigenvector of $\bV$, it is $ |\zeta^i_1 |< 1$. From $ \lambda_1 >0$, it is then $\rho(I +\eta \lambda_1(I-Z_1^2) \mathds{1} \mathds{1}^T)>1 $, meaning that when $ \bm{\xi}_1 = \epsilon \mathds{1}$ for some scalar coefficient $ \epsilon$ the system diverges, hence the polygonal equilibrium point $ \bs$ is unstable. 
\qed 

\begin{remark}
    In both Lemma~\ref{lem:moja-bip-consensus-stab} and Lemma~\ref{lem:moja-polygonal-stab} the perturbation $ \bw_i$ used in the proof is ``intrinsic'', i.e., it belongs to the tangent space $ T_{\bv_k}\mathbb{S}^{d-1} $ resp. $ T_{\bs_i}\mathbb{S}^{d-1} $ at the equilibrium point.
\end{remark}
For the multiagent Oja system \eqref{eq:update-moja} it is also possible to show that the trajectories are generically converging to an equilibrium point. 

\begin{lemma}
\label{lem:layp-moja}
The function $ W(\bx) = -\frac{1}{2 n^2} \left( \sum_{j=1}^n \bx_j^T \bV \sum_{j=1}^n \bx_j  \right)$ is a Lyapunov function for  \eqref{eq:update-moja}. When $  \eta\leq \frac{1}{2 \|\bV\|}  $ it guarantees that  \eqref{eq:update-moja} converges to one of the equilibria determined in Lemma~\ref{lem:moja-equil1}.
\end{lemma}
\proof
Since the state space is compact,  all functions arising in the expansion of the Lyapunov function $ W(\bx)  $ (and hence also of its increment) are uniformly bounded. For the second order truncation of the increment $ W(\bx^+) - W(\bx) $, it is possible to compute an explicit bound on $ \eta $ for which the increment is guaranteed to be nonincreasing.
From \eqref{eq:update-moja2},  using the binomial series expansion $ (1 + \xi)^{-\frac{1}{2}}\approx 1 - \frac{\xi}{2}  + \frac{3}{8} \xi^2 + O(\xi^3) $  
\[
\begin{split}
 & \bx_i^+  = \frac{\tilde{\bx}_i}{\| \tilde{\bx}_i \| } \\
 & = (\bx_i + \eta \by )  \left(1- \eta \bx_i^T \by  + \frac{\eta^2}{2} \left( 3 (\bx_i^T \by )^2 - \| \by\|^2   \right) + O(\eta^3)  \right) \\
 & = \bx_i +\eta g_i + \eta^2 \left( \frac{3}{2} r_i^2 \bx_i  - \frac{1}{2} \| \by\|^2 r_i - r_i \by \right) + O(\eta^3)  
 \end{split}
 \]
 where we have indicated $ r_i := \bx_i^T \by $ and $ g_i= g(\bx_i ) := (I - \bx_i \bx_i^T) \by $.
 The increment of $ W$ up to second order in $ \eta $ is then 
 \[
 \begin{split}
 & W(\bx^+) - W(\bx) = \\
 & =  -\frac{1}{2 n^2} \left( \sum_{j=1}^n (\bx_j^+)^T \bV \sum_{j=1}^n \bx_j^+  - \sum_{j=1}^n \bx_j^T \bV \sum_{j=1}^n \bx_j  \right) \\
 = & - \frac{\eta^2}{2 n^2} \left( 2 n \by^T \sum_j \left( \frac{3}{2} r_i^2 \bx_i - \frac{1}{2} \| \by\|^2 \bx_i - r_i \by \right) \right. \\
 & \left. + \sum_j g_j^T \bV \sum_j g_j \right) - \frac{\eta}{n} \sum_j g_j^T \by + O(\eta^3) \\
 = &  - \frac{\eta^2}{ n} \left(  \sum_j  \left( \frac{3}{2} r_j^2 \by^T \bx_j - \frac{1}{2} \| \by\|^2 \by^T \bx_j - r_j \by^T \by \right) \right. \\
 & \left. + \frac{1}{2n}  \sum_j g_j^T \bV \sum_j g_j \right)  - \frac{\eta}{n} \sum_j \| g_j \|^2 + O(\eta^3)  \\
 = & - \frac{\eta^2}{ n} \left( \frac{3}{2} \sum_j r_i  (r_i^2 - \| \by \|^2 ) + \frac{1}{2n}  \sum_j g_j^T \bV \sum_j g_j  \right) \\
 &  - \frac{\eta}{n} \sum_j \| g_j \|^2 + O(\eta^3)  \\
  = & \;  \frac{\eta^2}{ n} \left( \frac{3}{2} \sum_j r_i  \| g_j \|^2 - \frac{1}{2n}  \sum_j g_j^T \bV \sum_j g_j  \right) \\
 &  - \frac{\eta}{n} \sum_j \| g_j \|^2 + O(\eta^3)  \\
 \end{split}
 \]
 where we have used idempotency of the projection $ (I - \bx_i \bx_i^T)$ to get
 \[
 \begin{split}
 g_i^T \by & = \by^T (I - \bx_i \bx_i^T) \by \\
 & = \by^T (I - \bx_i \bx_i^T)(I - \bx_i \bx_i^T) \by = \| g_i\|^ 2\\
 \|\by\|^2 - r_i^2 & = \by^T \by - \by^T \bx_i \bx_i^T \by = \by^T (I - \bx_i \bx_i^T) \by \\
 & = \by^T (I - \bx_i \bx_i^T) (I - \bx_i \bx_i^T) \by = g_i^T g_i = \| g_i \|^2 .
 \end{split}
 \]
 Since $ \| \sum_j g_j \|^2 \leq n \sum_j \| g_j \|^2 $, it is
 \[
 - n \| \bV \| \sum_j \| g_j \|^2 \leq \sum_j g_j^T \bV \sum_j g_j \leq n \| \bV \| \sum_j \| g_j \|^2 .
 \]
  Furthermore
 \[
 - \| \bV \| \leq - \| \by \| \leq r_i \leq  \| \by \| \leq \| \bV \| .
 \]
 These lead to the inequality
 \[
 \begin{split}
  W(\bx^+)  & - W(\bx)   \leq  - \frac{\eta}{n} \sum_j \| g_j \|^2 \\
 & +\frac{\eta^2}{n} \left( \frac{3}{2} \| \bV \| \sum_j \| g_j \|^2 + \frac{1}{2} \| \bV \| \sum_j \| g_j \|^2 \right) 
 \end{split}
 \]
 or
 \[
 W(\bx^+) - W(\bx) \leq \frac{\eta}{n}  \left(  - 1 + 2 \eta \| \bV \| \right) \sum_j \| g_j \|^2.
 \]
 Hence when $ \eta\leq \frac{1}{2 \|\bV\|} $ we have $  W(\bx^+) - W(\bx) \leq 0$. 
 Since the sequence $ W(\bx (k) ) $ (where $k$ is the time index) is nonincreasing and bounded below, it must converge. Moreover, summability of the increments implies $ \sum_{k=0}^\infty \sum_{i=1}^n \| g(\bx_i(k)) \|^2 < \infty $, hence $  \| g(\bx_i(k)) \| \to 0 $ for each $i$. 
 At an equilibrium point $ \bx $ it must then be $ g(\bx_i) = (I - \bx_i \bx_i) \by =0 $ for all $ i$, i.e., $ \by $ must be collinear with $ \bx_i $. Recalling Lemma~\ref{lem:moja-equil1}, this implies that $ \bx $ must be a consensus or a bipartite consensus equilibrium point. 
 \qed

\noindent {\bf Proof of Theorem~\ref{thm:moja-DT}}
The proof of the theorem is simply the composition of the results of Lemmas~\ref{lem:moja-F-eigenvalues},~\ref{lem:moja-consensus-stab},~\ref{lem:moja-bip-consensus-stab},~\ref{lem:moja-polygonal-stab} and~\ref{lem:layp-moja}.
Lemmas~\ref{lem:moja-consensus-stab},~\ref{lem:moja-bip-consensus-stab},~\ref{lem:moja-polygonal-stab} say that only the consensus point $ \bv_1 $ corresponding to the principal eigenvalue of $\bV$ is locally asymptotically stable, while all other equilibria are unstable. 
Since all trajectories of \eqref{eq:update-moja} have a limit point, generically such limit point must be $ \bv_1$. 
\qed

\begin{remark}
The condition $ \eta< \min \left(\frac{2}{ |\lambda_1 + \lambda_d |}, \, \frac{1}{2 \max \{  \lambda_1 , | \lambda_d | \} }  \right) $ is sufficient but not necessary. 
Recall from Lemma~\ref{lem:moja-consensus-stab} that the condition $\eta< \frac{2}{ |\lambda_1 + \lambda_d |} $ is needed only when $ | \lambda_d|> \lambda_1$. 
\end{remark}

\section{Single-head self-attention dynamics}
\label{sec:self-att}

Let us now consider the self-attention system \eqref{eq:update}, which includes the attention matrix $ A(\bx)$. For this model, an ``individual'' weighted average with state dependent weights given by the attention coefficients, $ \bm{m}_i =\bm{m}_i(\bx)  =   \sum_{j=1}^n  A_{ij}(\bx) \bx_j $ replaces the simple average used in the multiagent Oja model. 
Consequently, the total action of all agents on agent $i$ varies from agent to agent:  $\by_i = \by_i (\bx) =  \bV\sum_{j=1}^n A_{ij}(\bx)   \bm{\bx}_j $. 

\subsection{Equilibria}

In addition to the three classes of equilibria already obtained for the multiagent Oja flow, in the self-attention dynamics we have an extra class, due to the fact that  $\by_i $ now differs from agent to agent.

\begin{lemma}
\label{lem:self-equil1}
The system \eqref{eq:update} has the following classes of equilibria
\benu
\item consensus: $ \bx_i = \bv_k $ $ \forall \, i$, and $ k=1, \ldots, d$;
\item bipartite consensus: $ \bx_i = \pm \bv_k $ $ \forall \, i$ and $ k=1, \ldots, d$;
\item polygonal equilibria: $ \{ \bx_i \in \mathbb{S}^{d-1} \; \text{s.t.} \;  \bV \sum_{j=1}^n A_{ij}(\bx) \bx_j =0\} $, $ i=1, \ldots, n$;
\item clustering equilibria: $  \{ \bx_i \in \mathbb{S}^{d-1} \; \text{s.t.} \;  \varphi_i \bx_i  = \bV \sum_{j=1}^n A_{ij}(\bx) \bx_j   \}$ for some scalars $ \varphi_i$, $ i=1, \ldots, n$, $ k=1, \ldots, d$.
\eenu

\end{lemma}
\proof
Once we replace $ \by $ with $\by_i  =  \sum_{j=1}^n A_{ij}(\bx)  \bV \bm{\bx}_j $, the proof of the first three cases is identical to that of Lemma~\ref{lem:moja-equil1}. 
The clustering equilibria correspond to $ \by_i \in \ker(I - \bx_i \bx_i^T) $, or, equivalently, $ \by_i $ aligned with $ \bx_i $, $ i=1, \ldots, n$, but not necessarily aligned with an eigenvector $ \bv_k$, i.e., $ \by_i = \varphi_i \bx_i $ for some scalar $ \varphi_i$, but possibly $ \by_i \nparallel  \bv_k $ $ \forall \, k$.  
\qed

\begin{remark}
Clustering equilibria own their name to the fact that typically multiple agents are found at the same value. In particular we say that $ \bx_1, \ldots, \bx_n$ are at an $ m$-clustering equilibrium point if $ \bx_i \in \{ \bm{w}_1, \ldots, \bm{w}_m \}$ with $ 1 \leq m \leq n$, i.e., if $ \bx_1, \ldots, \bx_n$ cluster at the $m$ vectors $ \bm{w}_1, \ldots, \bm{w}_m$. Notice that some $ \bm{w}_i $ can be eigenvectors of $\bV$.
\end{remark}

\begin{remark}
A clustering equilibrium can be a consensus, i.e, $ \bx_i  = \bx_j $ for all $ i,j$, but with $ \bx_i $ which is not aligned with any eigenvector of $\bV$, i.e., $ \bx_i \nparallel  \bv_k $ $ \forall \, k$. We refer to these as 1-clustering, while the ``consensus'' characterization is reserved for the case $ \bx_i \parallel \bv_k $ for some $k$. A 2-clustering instead typically is s.t. $ \bx_i \in \{ \bm{w}_1, \bm{w}_ 2\}$ $ \forall \, i$, with $ \bm{w}_1 \neq - \bm{w}_2 $.
\end{remark}

\begin{proposition}
\label{prop:clustering-eq}
The system  \eqref{eq:update} has a clustering equilibrium $ \bx \in (\mathbf{S}^{d-1})^n $ iff $ \exists $ scalars $ \theta_1, \ldots, \theta_n $ s.t. the matrix $I_n \otimes I_d  -(\Theta \otimes I_d) (A(\bx) \otimes I_d)(I_n\otimes \bV)$ is singular, where $ \Theta=\diag(\theta_1, \ldots, \theta_n )$. 
\end{proposition}
\proof 
In vector form, the clustering equilibrium condition $ \bx_i =  \theta_i  \bV \sum_{j=1}^n A_{ij}(\bx)  \bm{\bx}_j $ becomes $ \bx = (\Theta \otimes I_d) (A(\bx) \otimes I_d)(I_n\otimes \bV)  \bx$, where $ \bx^T = [ \bx_1^T \, \ldots\, \bx_n^T ]$. 
Rewriting it as $  (I_n \otimes I_d -(\Theta \otimes I_d) (A(\bx) \otimes I_d)(I_n\otimes \bV)) \bx =0 $, the algebraic equation has nontrivial solutions if and only if the matrix $I_n \otimes I_d  -(\Theta \otimes I_d) (A(\bx) \otimes I_d)(I_n\otimes \bV)$ is singular.
\qed

\begin{remark}
From the proof of Proposition~\ref{prop:clustering-eq}, it is $ \bx = (\Theta \otimes I_d) (A(\bx) \otimes I_d)(I_n\otimes \bV)  \bx$, where $ \bx^T = [ \bx_1^T \, \ldots\, \bx_n^T ]$. It follows that if $ \bx_1, \ldots, \bx_n$ form an $m$-clustering equilibrium point, then also the antipodal point  $ -\bx_1, \ldots, -\bx_n $ is an $ m$-clustering equilibrium point because $ A(\bx) = A(-\bx)$. If in a clustering equilibrium only some (but not all) $ \bx_i $ flip sign, then the new state is typically no longer an equilibrium point.
\end{remark}

While the attention matrix $ A(\bx) $ is a function of the state even at the equilibrium point (with the exception of a consensus state), its rank is however fixed for various classes of equilibria. 

\begin{lemma}
\label{lem:cons-Aij}
For a consensus state $ \bx_i = \bv_k $ $ \forall \, i$ we have $ A_{ij}(\bx) = \frac{1}{n} $, i.e., the attention matrix $ A(\bx) $ is the rank-1 uniform matrix $ A(\bx)  =\frac{1}{n} \mathds{1} \mathds{1}^T $. For a bipartite consensus state $ \bx_i  = \pm \bv_k $ $ \forall \, i$, the attention matrix $A(\bx)$ has rank 2.
For an $m$-clustering equilibrium $ \bx$ the rank of $ A(\bx)$ is $m$. \end{lemma}
\proof
At a consensus equilibrium $ \bx_i = \bv_k $ for all $ i $, and $ A_{ij}(\bx) $ is composed of all equal terms
\[ 
A_{ij}(\bx) = \frac{e^{  \bv_k^T \bQ^T\bK\bv_k  } }{\sum_{\ell=1}^n e^{\bv_k^T \bQ^T\bK \bv_k } }= \frac{1}{n}.
\]
At a bipartite consensus point $ \bx_i = \pm \bv_k $, split the $n$ tokens into two sets $ \mathcal{V}_1 $ and $ \mathcal{V}_2 $, $ \mathcal{V}_1\cup \mathcal{V}_2 = \{ 1, \ldots, n \}$,  according to whether $ \bx_i =\bv_k $ or $ \bx_i =- \bv_k$ at equilibrium. Assume that $ n_1 = | \mathcal{V}_1|  $ tokens are equal to $ \bv_k $ and $ n_2 = | \mathcal{V}_2 | $ equal to $ -\bv_k$, with $ n_1 + n_2 =n$.

Four different entries appear in $ A_{ij}(\bx)$, two due to the denominator (depending on whether $ \bx_i =\bv_k $ or $ \bx_i =- \bv_k $), and two due to the numerator (depending on whether $ \bx_i$ and $ \bx_j $ have the same sign). 
More specifically, denoting  $ \alpha_1^k = e^{  \bv_k^T \bQ^T\bK \bv_k }  $ and $ \alpha_2^k = e^{ - \bv_k^T \bQ^T\bK \bv_k } $, then the entries of the attention matrix are 
\[ 
A_{ij}(\bx) 
= \begin{cases} 
\frac{\alpha_1^k}{n_1 \alpha_1^k + n_2 \alpha_2^k} & \text{if}\; i \in \mathcal{V}_1, j \in \mathcal{V}_1 \\ 
\frac{\alpha_2^k}{n_1 \alpha_1^k + n_2 \alpha_2^k}  & \text{if}\; i \in \mathcal{V}_1, j \in \mathcal{V}_2 \\ 
\frac{\alpha_2^k}{n_1 \alpha_2^k + n_2 \alpha_1^k}   & \text{if}\; i \in \mathcal{V}_2, j \in \mathcal{V}_1 \\ 
\frac{\alpha_1^k}{n_1 \alpha_2^k + n_2 \alpha_1^k}   & \text{if}\; i \in \mathcal{V}_2, j \in \mathcal{V}_2  .
   \end{cases}
\]
Letting $ \beta_1^k = n_1 \alpha_1^k + n_2 \alpha_2^k$, and $ \beta_2^k = n_1 \alpha_2^k + n_2 \alpha_1^k$ and assuming w.l.o.g. that the first $ n_1 $ agents are in $ \mathcal{V}_1 $ and the last $ n_2 $ in $ \mathcal{V}_2 $, the attention matrix is 
\beq
A(\bx) = \begin{bmatrix} \frac{1}{\beta_1^k} ( \alpha_1^k & \ldots & \ldots & \alpha_1^k & \alpha_2^k & \ldots & \alpha_2^k) \\
\vdots & & & \vdots & \vdots & & \vdots \\
\frac{1}{\beta_1^k} ( \alpha_1^k & \ldots & \ldots & \alpha_1^k & \alpha_2^k & \ldots & \alpha_2^k) \\
\frac{1}{\beta_2^k} ( \alpha_2^k &  \ldots &  \ldots & \alpha_2^k & \alpha_1^k & \ldots &\alpha_1^k) \\
\vdots & & & \vdots  & \vdots & &  \vdots \\
\frac{1}{\beta_2^k} ( \alpha_2^k & \ldots &  \ldots & \alpha_2^k & \alpha_1^k & \ldots & \alpha_1^k) 
\end{bmatrix} ,
\label{eq:A-bipart}
\eeq
from which it is obvious that the rank must be 2. 

As for an $m$-clustering equilibrium point with vectors $ \bm{w}_1, \ldots, \bm{w}_m$, following a similar procedure it is easy to realize that since $ \bm{w}_i \neq \bm{w}_j $ and $ \bQ^T \bK$ is not symmetric,  there are $ m^2 $ different entries in the numerators of $A_{ij}$, and only $m$ in the denominators. Rearranging the entries as in \eqref{eq:A-bipart}, each row of $A$ has exactly $m$ different terms, and a block counting argument leads to $ \rank(A)=m$.
\qed

\subsection{Stability analysis}

As in the proof of Lemma~\ref{lem:cons-Aij}, let $ \alpha_1^k = e^{  \bv_k^T \bQ^T\bK \bv_k }  $, $ \alpha_2^k = e^{ - \bv_k^T \bQ^T\bK \bv_k } $ and $ \beta_1^k = n_1 \alpha_1^k + n_2 \alpha_2^k$, $ \beta_2^k = n_1 \alpha_2^k + n_2 \alpha_1^k$. Denote also $ \gamma_1^k = \frac{n_1 \alpha_1^k - n_2 \alpha_2^k }{\beta_1^k} $ and $ \gamma_2^k =  \frac{n_2 \alpha_1^k - n_1 \alpha_2^k }{\beta_2^k} $.

\begin{theorem}
\label{thm:stab-main-self}
For the self-attention dynamics~\eqref{eq:update}, under Assumption~\ref{ass:V-sym}, and, if $ -\lambda_d > \lambda_1 $, for $ \eta<  \frac{2}{ |\lambda_1 + \lambda_d |} $, the consensus equilibrium associated to the principal eigenvector $ \bv_1 $ of $\bV$ is always locally asymptotically stable, while the other consensus equilibria $ \bv_k$, $ k=2, \ldots, d $, are all unstable. 
A bipartite consensus equilibrium $ \bx_i = \pm \bv_k$ $ \forall \, i$ is asymptotically stable if $ \eta< \min\left( \frac{1}{\lambda_1 } , \, \frac{1}{ |\lambda_d |} \right) $ and the following inequalities are all satisfied
\benu
\item  
$ \frac{ |1 - (1-\eta \lambda_k \gamma_i^k)  \eta  \lambda_k\gamma_i^k | } {1+ \eta \lambda_k \gamma_i^k} <1 $, $ i =1,2 $;  
\item  $  a_j^k d_j^k - b_j^k c_j^k < 1 $ and $ a_j^k + d_j^k <1+ a_j^k d_j^k - b_j^k c_j^k $, $\, \forall \, j=1, \ldots, k-1, k+1,\ldots,  d $, where
 \beq
 \begin{split}
 a_j^k & = \frac{ 1+ \eta n_ 1 \frac{\alpha_1^k}{\beta_1^k}  \lambda_j  }{1+ \eta \lambda_k \gamma_1^k} , \qquad 
 b_j^k  = \frac{ \eta n_2 \frac{\alpha_2^k}{\beta_1^k}  \lambda_j  }{1+ \eta \lambda_k \gamma_1^k}, \\
 c_j^k  & = \frac{ \eta n_1 \frac{\alpha_2^k}{\beta_2^k}   \lambda_j  }{1+ \eta \lambda_k \gamma_2^k} , \qquad 
 d_j^k  =\frac{  1+  \eta n_ 2 \frac{\alpha_1^k}{\beta_2^k}   \lambda_j  }{1+ \eta \lambda_k \gamma_2^k  }
 \end{split}
 \label{eq:a-d} 
 \eeq
 \item $  \frac{1}{1+ \eta \lambda_k \gamma_i^k} <1$, $ i=1,2$.
 \eenu
All polygonal equilibria are unstable.
\end{theorem}

The proof, based on Lyapunov indirect method, is broken down into various lemmas. 
As in Section~\ref{sec:multi-oja}, let $ f_i(\bx)   $ be the right-hand side of \eqref{eq:update} and $ f(\bx)  $ its vectorization. Let also $ \chi(\bx) := \|  \bx_i  +  \eta \bV \sum_j A_{ij}(\bx) \bx_j \| $. The (extrinsic, see Remark~\ref{rem:two-classes}) Jacobian of $ f(\bx)  $  is as follows.

\begin{lemma}
\label{lem:Jacob-self}
The Jacobian of  \eqref{eq:update}, $ J(\bx) = \pde{f(\bx)}{\bx}= \begin{bmatrix}\pde{f_i(\bx)}{\bx_h} \end{bmatrix} $, has the following blocks:
 \beq
 \label{eq:jacon-atten1}
 \begin{split}
&  \pde{f_i(\bx)}{\bx_h} =   
 \frac{I \delta_{ih} + \eta \bV \left( \sum_j  \bx_j \pde{A_{ij} (\bx) }{\bx_h}   + A_{ih}(\bx) I\right) }{\chi(\bx) } \\
& \quad  -  \frac{  \left(  \bx_i  + \eta \bV \sum_j A_{ij}(\bx) \bx_j \right) \pde{}{\bx_h} \chi(\bx) } {\chi(\bx)^2 }
\end{split}
 \eeq
 where
 \beq
 \begin{split}
 &\pde{A_{ij} (\bx) }{\bx_h} 
  =  \Bigg(-  \bx_i^T \bQ^T \bK A_{ih}(\bx)  +  \bx_i^T \bQ^T \bK  \delta_{jh}  \\
  & + \left( \bx_j^T \bK^T \bQ  - \sum_\ell \bx_\ell^T \bK^T \bQ A_{i\ell} (\bx) \right) \delta_{ih} \Bigg) A_{ij} (\bx)
  \end{split}
 \label{eq:partial-A}
 \eeq
and
\beq
\label{eq:Jac-attent1}
\begin{split}
&\pde{\chi(\bx)}{\bx_h } =  \left( \sum_j\bx_j  \pde{A_{ij}(\bx)}{\bx_h} + A_{ih}(\bx) I\right)  \cdot \\
&  \cdot \left( \eta \bx_i^T \bV  + \eta^2 \sum_j \bx_j^T A_{ij}(\bx) \bV^T (I - \bx_i \bx_i^T)  \bV   \right)  
\\
& + \left( \left(  \eta - \eta^2    \bx_i^T  \bV \sum_j A_{ij}(\bx) \bx_j \right)  \sum_j A_{ij}(\bx) \bx_j^T \bV^T  \right) \delta_{ih}
\end{split}
\eeq
\end{lemma}

\proof
The partial derivatives of $ A_{ij}(\bx) $ are 
 \[
 \begin{split}
 \pde{A_{ij} (\bx) }{\bx_h} 
  = & \pde{}{\bx_h} \frac{e^{\langle \bQ \bx_i, \bK \bx_j \rangle } }{\sum_{\ell=1}^n e^{\langle \bQ \bx_i, \bK \bx_\ell \rangle } }\\
 =  & -  \bx_i^T \bQ^T \bK A_{ij} (\bx) A_{ih}(\bx) \\
 & +  \bx_i^T \bQ^T \bK A_{ij} (\bx)\delta_{jh} \\
&  + \bx_j^T \bK^T \bQ A_{ij} (\bx)\delta_{ih}  \\
& - \sum_\ell \bx_\ell^T \bK^T \bQ A_{i\ell} (\bx)A_{ij}(\bx) \delta_{ih} 
 \end{split}
 \]
 or \eqref{eq:partial-A} after regrouping.
Similarly to the proof of Lemma~\ref{lem:Jac-moja},  the series expansion of the normed term is
\[
\begin{split}
& \chi( \bx) =  \left(  (  \bx_i +  \eta \by_i )^T (  \bx_i + \eta \by_i) \right)^{\frac{1}{2}}  
\\
& \approx  1 + \frac{1}{2} \left( 2  \eta \bx_i ^T \by_i + \eta^2 \| \by_i \|^2  \right) - \frac{1}{2}  \eta^2 ( \bx_i^T \by_i ) ^2 + O(\eta^3) \\
& =  1 + \eta \bx_i^T \bV \sum_j A_{ij}(\bx) \bx_j + \frac{\eta^2}{2} \left( \sum_j A_{ij}(\bx) \bx_j^T \bV^T \bV \cdot \right. \\
& \left. \cdot \sum_j A_{ij}(\bx) \bx_j - ( \bx_i^T \bV \sum_j A_{ij}(\bx) \bx_j)^2 \right)  + O(\eta^3)
\end{split}
\]
Differentiating it
\[
\begin{split}
& \pde{\chi(\bx)}{\bx_h } =   \,  \eta \,  \bx_i^T \bV \left(  \sum_j \bx_j \pde{A_{ij}(\bx)}{\bx_h} + A_{ih}(\bx) I \right) \\
&   + \eta^2 \left(  \sum_j \bx_j^T A_{ij}(\bx) \bV^T \bV  \left(  \sum_j \bx_j\pde{A_{ij}(\bx)}{\bx_h}  + A_{ih}(\bx)I \right) 
 \right. \\
&  \left. -  \left( \bx_i^T \bV \sum_j A_{ij}(\bx) \bx_j \right) \bx_i^T \bV \right. \cdot \\
& \cdot \left.  \left(  \sum_j\bx_j \pde{A_{ij}(\bx)}{\bx_h}  + A_{ih}(\bx) I\right) \right) + \left( \eta  \sum_j \bx_j^T A_{ij}(\bx) \bV^T \right. \\
& \left. - \eta^2  \left( \bx_i^T \bV \sum_j A_{ij}(\bx) \bx_j \right)\sum_j \bx_j^T A_{ij}(\bx) \bV^T \right) \delta_{ih}
\end{split}
\]
or \eqref{eq:Jac-attent1} after regrouping.
One gets the result inserting these quantities into \eqref{eq:jacon-atten1}.
 \qed

The following lemma states that the behavior of the self-attention model \eqref{eq:update} at a consensus point is locally identical to that of the multiagent Oja model \eqref{eq:update-moja}.
\begin{lemma}
\label{lem:lin-stab-self1}
At a consensus point $ \bv_k$, the eigenvalues of the Jacobian matrix $ J(\bv_k) $ are those indicated in Lemma~\ref{lem:moja-F-eigenvalues}.
In particular, under Assumption~\ref{ass:V-sym}, and, if $ -\lambda_d > \lambda_1 $, for $ \eta<  \frac{2}{ |\lambda_1 + \lambda_d |} $, the consensus point $ \bx_i =\bv_1$ $ \forall\, i$ is locally asymptotically stable for \eqref{eq:update}, while the remaining consensus equilibria $ \bx_i= \bv_k$ $ \forall \, i $, with $ k= 2, \ldots, d $, are all unstable. 
\end{lemma}

\proof
From Lemma~\ref{lem:cons-Aij}, in a consensus equilibrium the attention matrix is the uniform matrix $ A(\bm{\bv_k}) =\frac{1}{n} \mathds{1}\mathds{1}^T$.
It follows that the Jacobian at $ \bx_i = \bv_k $, $ J(\bv_k) $, is still given by \eqref{eq:JacobF-moja}, and hence, from Lemma~\ref{lem:moja-consensus-stab}, that $ \bv_1 $ is locally asymptotically stable while $ \bv_k $, $k=2, \ldots, d $, are all unstable.
\qed

We consider now a bipartite consensus equilibrium and assume w.l.o.g. that the first $ n_1 $ agents are in $ \mathcal{V}_1 $ and the last $ n_2 $ in $ \mathcal{V}_2 $.

\begin{lemma}
\label{lem:Jac-bip-cons}
At a bipartite consensus point $ \pm  \bv_k $, the Jacobian $ J(\pm\bv_k) $ is given by 
\[
\begin{split}
J(\pm\bv_k) & = 
 \begin{bmatrix} 
J_o^{11} & \ldots &  J_o^{11} &  J_o^{12}   & \ldots &  J_o^{12} \\ 
\vdots & & \vdots & \vdots & & \vdots \\
J_o^{11} & \ldots &  J_o^{11} &  J_o^{12}   & \ldots &  J_o^{12} \\ 
J_o^{21} & \ldots &  J_o^{21} &  J_o^{22}   & \ldots &  J_o^{22} \\ 
\vdots & & \vdots & \vdots & & \vdots \\
J_o^{21} & \ldots &  J_o^{21} &  J_o^{22}   & \ldots &  J_o^{22} \\ 
 \end{bmatrix} \\
&  +
  \begin{bmatrix} J_d^1 & \\ &  \ddots \\ & &  J_d^1 \\
& & &  J_d^2 \\ & & & & \ddots \\ & & & & &  J_d^2   \end{bmatrix} 
,
 \end{split}
\]
where
\beq\label{eq:J_bip}
 \begin{split}
 J_o^{11}& =  \frac{
  \eta  \frac{\alpha_1^k}{\beta_1^k}}
{1+ \eta \lambda_k \gamma_1^k}  \Big(\sigma_1  \bV - \lambda_k \bv_k \bv_k^T  \\
& \quad -   ( 1-\sigma_1) (1-\gamma_1^k) \lambda_k \bv_k \bv_k^T \bQ^T \bK  \Big) \\
 J_o^{12} & =  \frac{
  \eta  \frac{\alpha_2^k}{\beta_1^k} }
{1+ \eta \lambda_k \gamma_1^k}  \Big( \sigma_1 \bV - \lambda_k \bv_k \bv_k^T   \\
& \quad + (1- \sigma_1) (1+\gamma_1^k) \lambda_k \bv_k \bv_k^T \bQ^T \bK  \Big) 
 \\
  J_o^{21} &  =  \frac{
  \eta  \frac{\alpha_2^k}{\beta_2^k} }
{1+ \eta \lambda_k \gamma_2^k}  \Big( \sigma_2 \bV - \lambda_k \bv_k \bv_k^T \\
& \quad   +  (1-\sigma_2) (1+\gamma_2^k) \lambda_k \bv_k \bv_k^T \bQ^T \bK  \Big) 
 \\
   J_o^{22} & =    \frac{
  \eta  \frac{\alpha_1^k}{\beta_2^k} }
{1+ \eta \lambda_k \gamma_2^k}  \Big( \sigma_2 \bV - \lambda_k \bv_k \bv_k^T   \\
& \quad  - (1-\sigma_2) (1-\gamma_2^k) \lambda_k \bv_k \bv_k^T \bQ^T \bK \Big) 
 \\
 J_d^1 &  = \frac{\sigma_1 I - ( 1- \sigma_1) \eta (1-(\gamma_1^k)^2) \lambda_k \bv_k \bv_k^T \bQ^T \bK}
 {1+ \eta \lambda_k \gamma_1^k} \\
 & \quad - \frac{  (1-\eta \lambda_k \gamma_1^k)  \eta  \lambda_k \gamma_1^k  \bv_k \bv_k^T }
 {1+ \eta \lambda_k \gamma_1^k} \\
  J_d^2 &  = \frac{\sigma_2 I -(1-\sigma_2)  \eta (1-(\gamma_2^k)^2) \lambda_k \bv_k \bv_k^T \bQ^T \bK }
{1+ \eta \lambda_k \gamma_2^k} \\
& \quad  -\frac{ (1-\eta \lambda_k \gamma_2^k) \eta  \lambda_k \gamma_2^k  \bv_k \bv_k^T }
{1+ \eta \lambda_k \gamma_2^k} \\
 \end{split}
 \eeq
 with $ \sigma_1 = \sign({1+ \eta \lambda_k \gamma_1^k}) $ and $ \sigma_2 = \sign({1+ \eta \lambda_k \gamma_2^k}) $.
 
\end{lemma}

\proof
Some tedious but elementary calculations yield:
 \[
 \left.  \bV \sum_j A_{ij}(\bx) \bx_j  \right|_{\bx_{i,j} = \pm \bv_k} = \begin{cases} 
 \gamma_1^k \lambda_k \bv_k &  \text{ if } \; \; i \in \mathcal{V}_1 \\
- \gamma_2^k \lambda_k \bv_k &  \text{ if } \; \; i \in \mathcal{V}_2 
\end{cases}
\]
\[
\begin{split}
 &\left. \bx_i^T  \bV \sum_j A_{ij}(\bx) \bx_j  \right|_{\bx_{i,j} = \pm \bv_k} = \\ 
& \qquad \qquad = \begin{cases} \bv_k^T \bV \gamma_1^k \bv_k = \lambda_k \gamma_1^k &\!\!\!\! \text{ if } \;  i \in \mathcal{V}_1 \\
  - \bv_k^T \bV (- \gamma_2^k) \bv_k =  \lambda_k \gamma_2^k &\!\!\!\!  \text{ if } \;  i \in \mathcal{V}_2 
  \end{cases}
  \end{split}
  \]
  Since $ \bv_k^T (I - \bv_k \bv_k^T) =0 $, it is
\[
\left. \sum_j \bx_j^T A_{ij}(\bx) \bV^T (I - \bx_i \bx_i^T)  \bV \right|_{\bx_{i,j} = \pm \bv_k} =0 .
\]
For the derivatives of the attention matrix, we have
\[
\begin{split}
& \pde{A_{ij}(\bx)}{\bx_h} =\\
& = \begin{cases}
 \bullet\;  \text{if } \; \; i  \in \mathcal{V}_1, \; h \in \mathcal{V}_1, \;  j \in \mathcal{V}_1 : \\
 \left( - \frac{\alpha_1^k}{\beta_1^k} + \delta_{jh}\right) \frac{\alpha_1^k}{\beta_1^k} \bv_k^T \bQ^T \bK 
+ \left(( 1 -\gamma_1^k) \frac{\alpha_1^k}{\beta_1^k}   \bv_k^T \bK^T \bQ \right) \delta_{ih} \\  
\bullet\;  \text{if } \; \; i \in \mathcal{V}_1 ,\;  h \in \mathcal{V}_1 , \; j\in \mathcal{V}_2 \\
- \frac{\alpha_1^k}{\beta_1^k}  \frac{\alpha_2^k}{\beta_1^k}  \bv_k^T \bQ^T \bK 
-( 1 +\gamma_1^k) \frac{\alpha_2^k}{\beta_1^k}  \left( \bv_k^T \bK^T \bQ \right) \delta_{ih} \\
\bullet\;  \text{if } \; \; i \in \mathcal{V}_1, \; h\in \mathcal{V}_2, \;  j \in \mathcal{V}_1 \\
- \frac{\alpha_2^k}{\beta_1^k} \frac{\alpha_1^k}{\beta_1^k} \bv_k^T \bQ^T \bK \\
\bullet\;  \text{if }  \; \; i \in \mathcal{V}_1, \; h\in \mathcal{V}_2 , \;  j \in \mathcal{V}_2  \\
\left( - \frac{\alpha_2^k}{\beta_1^k} + \delta_{jh}\right) \frac{\alpha_2^k}{\beta_1^k} \bv_k^T \bQ^T \bK \\
\bullet\;  \text{if }  \; \; i \in \mathcal{V}_2, \; h \in \mathcal{V}_1, \; j \in \mathcal{V}_1  \\
\left(  \frac{\alpha_2^k}{\beta_2^k} - \delta_{jh}\right) \frac{\alpha_2^k}{\beta_2^k} \bv_k^T \bQ^T \bK \\
\bullet\;  \text{if }  \; \; i \in \mathcal{V}_2, \; h \in \mathcal{V}_1 , \;  j \in \mathcal{V}_2  \\
  \frac{\alpha_2^k}{\beta_2^k}  \frac{\alpha_1^k}{\beta_2^k} \bv_k^T \bQ^T \bK \\
\bullet\;  \text{if } \; \; i  \in \mathcal{V}_2, \; h \in \mathcal{V}_2, \;  j \in \mathcal{V}_1  \\
 \frac{\alpha_1^k}{\beta_2^k}  \frac{\alpha_2^k}{\beta_2^k} \bv_k^T \bQ^T \bK 
+ \left(( 1 +\gamma_2^k) \frac{\alpha_2^k}{\beta_2^k}   \bv_k^T \bK^T \bQ \right) \delta_{ih} \\
\bullet\; \text{if } \; \; i  \in \mathcal{V}_2, \; h \in \mathcal{V}_2, \;  j \in \mathcal{V}_2 \\
\left(  \frac{\alpha_1^k}{\beta_2^k} - \delta_{jh} \right) \frac{\alpha_1^k}{\beta_2^k} \bv_k^T \bQ^T \bK  - \left(( 1 -\gamma_2^k) \frac{\alpha_1^k}{\beta_2^k}   \bv_k^T \bK^T \bQ \right) \delta_{ih} \\
\end{cases}
\end{split}
\]
and hence
\[
\begin{split}
&\sum_j \bx_j  \pde{A_{ij}(\bx)}{\bx_h} = \\
& = \begin{cases}
\bullet\;  \text{if } \; \; i  \in \mathcal{V}_1, \; h \in \mathcal{V}_1\\
\left(  \frac{\alpha_1^k}{\beta_1^k} (1 - \gamma_1^k ) + (1-(\gamma_1^k)^2 ) \delta_{ih} \right) \bv_k  \bv_k^T \bQ^T \bK  \\
\bullet\; \text{if } \; \; i  \in \mathcal{V}_1, \; h \in \mathcal{V}_2\\ 
-  \frac{\alpha_2^k}{\beta_1^k} (1 + \gamma_1^k ) \bv_k \bv_k^T \bQ^T \bK   \\
\bullet\;  \text{if } \; \; i  \in \mathcal{V}_2, \; h \in \mathcal{V}_1\\
-\frac{\alpha_2^k}{\beta_2^k} (1 + \gamma_2^k ) \bv_k  \bv_k^T \bQ^T \bK  \\
\bullet\;  \text{if } \; \; i  \in \mathcal{V}_2, \; h \in \mathcal{V}_2\\
\left(  \frac{\alpha_1^k}{\beta_2^k} (1 - \gamma_2^k ) + (1-(\gamma_2^k)^2 )  \delta_{ih} \right)  \bv_k \bv_k^T \bQ^T \bK   \\
\end{cases}
\end{split}
\]
Concerning the normed terms, 
\[
\begin{split}
 \left. \chi(\bx) \right|_{\bx_{i,j} = \pm \bv_k} = \begin{cases}
| 1+ \eta \lambda_k \gamma_1^k | & \text{if } \; \; i  \in \mathcal{V}_1 \\
| 1+ \eta \lambda_k \gamma_2^k | & \text{if } \; \; i  \in \mathcal{V}_2
\end{cases}
\end{split}
\]
and
\[
\label{eq:Jac-attent2}
\begin{split}
&\left. \pde{\chi(\bx)}{\bx_h }  \right|_{\bx_{i,h,j} = \pm \bv_k} =  \\
& = \begin{cases} 
\bullet\;   \text{if } \; \; i  \in \mathcal{V}_1, \; h \in \mathcal{V}_1\\
\eta \lambda_k \bv_k^T   \left( \frac{\alpha_1^k}{\beta_1^k} \left(  1+ (1 - \gamma_1^k ) \bQ^T \bK \right)  \right.   \\  \quad   + \left( (1-(\gamma_1^k)^2 ) \bQ^T \bK  +  (1-\eta \lambda_k \gamma_1^k ) \gamma_1^k  \right) \delta_{ih}  \Big) \\
\bullet\;  \text{if } \; \; i  \in \mathcal{V}_1, \; h \in \mathcal{V}_2\\
\eta \lambda_k \bv_k^T  \frac{\alpha_2^k}{\beta_1^k} \left( 1 - (1 + \gamma_1^k ) \bQ^T \bK  \right)    \\
\bullet\;  \text{if } \; \; i  \in \mathcal{V}_2, \; h \in \mathcal{V}_1\\
- \eta \lambda_k  \bv_k^T  \frac{\alpha_2^k}{\beta_2^k} \left( 1 - (1 + \gamma_2^k ) \bQ^T \bK  \right)    \\
\bullet\;  \text{if } \; \; i  \in \mathcal{V}_2, \; h \in \mathcal{V}_2\\
- \eta \lambda_k  \bv_k^T \left( \frac{\alpha_1^k}{\beta_2^k}  \left( 1 + (1 - \gamma_2^k )  \bQ^T \bK  \right)  \right. \\
\quad  + \left( (1-(\gamma_2^k)^2 ) \bQ^T \bK  + (1-\eta \lambda_k \gamma_2^k ) \gamma_2^k \right) \delta_{ih}  \Big) \\
\end{cases}
\end{split}
\]
Combining these terms together, we obtain the two components of $\pde{f_i(\bx)}{\bx_h} $
\[
\begin{split}
& \left.  \frac{I \delta_{ih} + \eta \bV \left( \sum_j  \bx_j  \pde{A_{ij} (\bx) }{\bx_h}  + A_{ih}(\bx) I \right) }{\chi(\bx) }\right|_{\bx_{i,j} = \pm \bv_k} = \\
&  =
 \begin{cases}
 \bullet\;    \text{if } \; \; i  \in \mathcal{V}_1, \; h \in \mathcal{V}_1\\
 \frac{\eta \frac{\alpha_1^k}{\beta_1^k} \left( \bV +  (1-\gamma_1^k)\lambda_k \bv_k \bv_k^T \bQ^T \bK \right) }
 {| 1+ \eta \lambda_k \gamma_1^k | } \\ 
 \qquad + \frac{\left( I + \eta  (1-(\gamma_1^k)^2) \lambda_k \bv_k \bv_k^T \bQ^T \bK  \right)\delta_{ih} }
 {| 1+ \eta \lambda_k \gamma_1^k | } \\
 \bullet\;    \text{if } \; \; i  \in \mathcal{V}_1, \; h \in \mathcal{V}_2\\
 \frac{\eta  \frac{\alpha_2^k}{\beta_1^k} \left(\bV - (1+\gamma_1^k) \lambda_k \bv_k \bv_k^T \bQ^T \bK  \right) }
 {| 1+ \eta \lambda_k \gamma_1^k | } \\
 \bullet\;  \text{if } \; \; i  \in \mathcal{V}_2, \; h \in \mathcal{V}_1\\
  \frac{\eta  \frac{\alpha_2^k}{\beta_2^k} \left( \bV - (1+\gamma_2^k) \lambda_k \bv_k \bv_k^T \bQ^T \bK  \right) }
 {| 1+ \eta \lambda_k \gamma_2^k | } \\
 \bullet\;    \text{if } \; \; i  \in \mathcal{V}_2, \; h \in \mathcal{V}_2\\
 \frac{\eta \frac{\alpha_1^k}{\beta_2^k} \left( \bV + (1-\gamma_2^k)\lambda_k \bv_k  \bv_k^T \bQ^T \bK \right) }
 {| 1+ \eta \lambda_k \gamma_2^k | } \\ 
 \qquad + \frac{\left( I + \eta  (1-(\gamma_2^k)^2) \lambda_k \bv_k \bv_k^T \bQ^T \bK \right)\delta_{ih}}
 {| 1+ \eta \lambda_k \gamma_2^k | } \\
 \end{cases}
 \end{split}
 \]
 \[
\begin{split}
 &\left.  \frac{  \left(  \bx_i  + \eta \bV \sum_j A_{ij}(\bx) \bx_j \right) \pde{}{\bx_h} \chi(\bx) } {\chi(\bx)^2 }\right|_{\bx_{i,j} = \pm \bv_k} \!\!\!\!\!\! \!\!\!\!\!\! = \\
 & \;\; 
 = \begin{cases}
 \bullet\;   \text{if } \; \; i  \in \mathcal{V}_1, \; h \in \mathcal{V}_1 \\
 \frac{\eta \lambda_k \bv_k \bv_k^T \Big( 
 \frac{\alpha_1^k}{\beta_1^k} \left( 1+ (1-\gamma_1^k)  \bQ^T \bK \right) }
{1+ \eta \lambda_k \gamma_1^k}  \\
 \qquad +  \frac{\left[ (1-(\gamma_1^k)^2) \bQ^T \bK  +(1-\eta \lambda_k \gamma_1^k) \gamma_1^k  \right] \delta_{ih} \Big) }
{1+ \eta \lambda_k \gamma_1^k} \\
 \bullet\;  \text{if } \; \; i  \in \mathcal{V}_1, \; h \in \mathcal{V}_2 \\
 \frac{\eta \lambda_k \bv_k \bv_k^T \frac{\alpha_2^k}{\beta_1^k} \left( 1- (1+\gamma_1^k)  \bQ^T \bK \right) }
{1+ \eta \lambda_k \gamma_1^k} \\
 \bullet\;   \text{if } \; \; i  \in \mathcal{V}_2, \; h \in \mathcal{V}_1 \\
 \frac{\eta \lambda_k \bv_k \bv_k^T \frac{\alpha_2^k}{\beta_2^k} \left( 1- (1+\gamma_2^k) \bQ^T \bK  \right) }
{1+ \eta \lambda_k \gamma_2^k} \\
 \bullet\;   \text{if } \; \; i  \in \mathcal{V}_2, \; h \in \mathcal{V}_2 \\
\frac{\eta \lambda_k \bv_k \bv_k^T \Big( \frac{\alpha_1^k}{\beta_2^k} \left( 1+ (1-\gamma_2^k) \bQ^T \bK \right) }
{1+ \eta \lambda_k \gamma_2^k} \\
 \qquad + \frac{\left[  (1-(\gamma_2^k)^2) \bQ^T \bK  + (1-\eta \lambda_k \gamma_2^k) \gamma_2^k  \right] \delta_{ih} \Big) }
{1+ \eta \lambda_k \gamma_2^k} \\
\end{cases}
\end{split}
\]
which lead to the expression for $ J(\pm \bv_k)$. 
\qed

To avoid dealing with absolute values in $ 1 + \eta \lambda_k \gamma_i  $, $ i=1,2$, in $ J(\pm \bv_k)$, it is convenient to restrict the step size $ \eta $ opportunely.
\begin{lemma}
\label{lemm:simplif-eta}
When $ \eta< \min\left( \frac{1}{\lambda_1 } , \, \frac{1}{ | \lambda_d |} \right) $ then $ 1+ \eta \lambda_k \gamma_i^k > 0$, $ i=1,2$. 
\end{lemma}
\proof
Notice first that, since $ |\gamma_i^k |\leq 1 $, $  | \lambda_k| \geq | \lambda_k \gamma_i^k | $, hence $\min\left( \frac{1}{\lambda_1 } , \, \frac{1}{ | \lambda_d |} \right)  \leq \frac{1}{ | \lambda_k| } \leq \frac{1}{ | \lambda_k \gamma_i^k |}$. 
Consequently, $ \eta< \min\left( \frac{1}{\lambda_1 } , \, \frac{1}{ | \lambda_d |} \right) $ implies $ \eta |  \lambda_k \gamma_i^k | \leq 1 $, or $ 1- \eta |  \lambda_k \gamma_i^k | >0$, which in turn implies $ 1 + \eta \lambda_k \gamma_i^k \geq 0 $, $ i=1,2$. 
\qed

This limitation on $ \eta$ implies $ \sigma_1 =\sigma_2 =1 $, which leads to a considerable simplification in the Jacobian terms of \eqref{eq:J_bip}, and therefore in the calculation of its eigenvalues.

\begin{lemma}
\label{lem:lin-stab-self2}
When $ \eta< \min\left( \frac{1}{\lambda_1 } , \, \frac{1}{ | \lambda_d |} \right) $, at a bipartite consensus point $ \pm  \bv_k $, the eigenvalues of the Jacobian $ J(\pm\bv_k) $ are 
 \benu
 \item 
  $ \frac{ 1 - (1-\eta \lambda_k \gamma_1^k)  \eta  \lambda_k\gamma_1^k  }
 {1+ \eta \lambda_k \gamma_1^k}\  $ of multiplicity $ n_1$, and  $ \frac{ 1 - (1-\eta \lambda_k \gamma_2^k)  \eta  \lambda_k\gamma_2^k  }
 {1+ \eta \lambda_k \gamma_2^k}\  $  of multiplicity $ n_2$.
 \item $ \mu_{j,\pm}^k =  \frac{1}{2} \Big( a_j^k+ d_j^k \pm \sqrt{ (a_j^k-d_j^k)^2 + 4 c_j^kb_j^k } \Big) $ for $ j=1, \ldots, k-1, k+1,\ldots, d $, where $ a_j^k$, $ b_j^k$, $ c_j^k $ and $ d_j^k$ are given by \eqref{eq:a-d}.

 \item   $  \frac{1}{1+ \eta \lambda_k \gamma_1^k} $ and $  \frac{1 }{1+ \eta \lambda_k \gamma_2^k} $, of total multiplicity $ nd -n -2d +2 $.
 \eenu

\end{lemma}
\proof
For $ J(\pm \bv_k) $ there are 3 classes of eigenvectors:
 \benu
 \item 
 If $ 1+ \eta \lambda_k \gamma_i^k >0$, $i=1, 2$, then choose $ \bm{p}^\ell = [ \,\underbrace{0 \, \ldots \, 0}_{\ell-1} \, \bv_k^T \, \underbrace{0 \, \ldots \, 0}_{n- \ell} \,]^T$. There are $ n$ such eigenvectors, and they are obviously all orthogonal to each other. For $\ell \leq n_1 $ it is 
  \[
 \begin{cases}
 J_o^{11} \bv_k & =  \frac{
  \eta  \frac{\alpha_1^k}{\beta_1^k} \left( \bV - \lambda_k \bv_k \bv_k^T \right)  
 }
{1+ \eta \lambda_k \gamma_1^k} \bv_k 
=  \frac{ \eta  \frac{\alpha_1^k}{\beta_1^k} \left( \lambda_k \bv_k  - \lambda_k \bv_k \right)}
{1+ \eta \lambda_k \gamma_1^k} =0
\\
   J_o^{21}  \bv_k &  =  \frac{
  \eta  \frac{\alpha_2^k}{\beta_2^k}  \left( \bV - \lambda_k \bv_k \bv_k^T \right)  
}
{1+ \eta \lambda_k \gamma_2^k} \bv_k = \frac{
  \eta  \frac{\alpha_2^k}{\beta_2^k}   \left( \lambda_k \bv_k  - \lambda_k \bv_k \right)  
}
{1+ \eta \lambda_k \gamma_2^k} =0
\\
 J_d^1 \bv_k &  = \frac{ I -(1-\eta \lambda_k \gamma_1^k)  \eta  \lambda_k \gamma_1^k \bv_k \bv_k^T }
 {1+ \eta \lambda_k \gamma_1^k} \bv_k \\
&  = \frac{ 1 -  (1-\eta \lambda_k \gamma_1^k) \eta  \lambda_k \gamma_1^k }
 {1+ \eta \lambda_k \gamma_1^k}\bv_k 
 \\
   \end{cases}
\]
meaning that 
\[
J(\pm \bv_k ) \bp =\frac{ 1 -(1-\eta \lambda_k \gamma_1^k) \eta \lambda_k  \gamma_1^k  }
 {1+ \eta \lambda_k \gamma_1^k}\ \bp
\]
i.e., $ \frac{ 1 - (1-\eta \lambda_k \gamma_1^k)  \eta  \lambda_k\gamma_1^k  }
 {1+ \eta \lambda_k \gamma_1^k}\  $ is an eigenvalue of $ J(\pm \bv_k)$ of multiplicity $ n_1$.
For $ \ell> n_1 $, similar calculations give
\[
 J_o^{12} \bv_k = J_o^{22} \bv_k =0, \;\; J_d^2 \bv_k = \frac{ 1 -  (1-\eta \lambda_k \gamma_2^k) \eta  \lambda_k \gamma_2^k }
 {1+ \eta \lambda_k \gamma_2^k}\bv_k 
 \]
 i.e., $ \frac{ 1 - (1-\eta \lambda_k \gamma_2^k)  \eta  \lambda_k\gamma_2^k  }
 {1+ \eta \lambda_k \gamma_2^k}\  $ is an eigenvalue of $ J(\pm \bv_k)$ of multiplicity $ n_2$.

\item Second class: Consider $ d-1$  vectors  $ \bm{q}^h $ s.t. $ \bm{q}^h =[\, \underbrace{ (\bm{w}_1^h)^T \, \ldots  \,  (\bm{w}_1^h)^T}_{n_1 \; \text{times}} \; \underbrace{(\bm{w}_2^h)^T \, \ldots  \, (\bm{w}_2^h)^T}_{n_2 \; \text{times}} ]^T $, with 
$ \bm{w}_1^h = \sum_{j=1}^d \xi_j^{h,1} \bv_j $ and $ \bm{w}_2^h = \sum_{j=1}^d \xi_j^{h,2} \bv_j $.
Choose $ \bw_i^h $ s.t. $ \xi_k^{h,i}=0$, so that $ \bv_k^T \bw_i^h =0$.
We have 
\[
\begin{split} 
J_o^{11} \bm{w}_1^h & = \frac{\eta  \frac{\alpha_1^k}{\beta_1^k}   \sum_{j\neq k } \lambda_j \xi_j^{h,1} \bv_j }{1+ \eta \lambda_k \gamma_1^k} \\
J_o^{12} \bm{w}_2^h & = \frac{\eta  \frac{\alpha_2^k}{\beta_1^k}  \sum_{j\neq k } \lambda_j \xi_j^{h,2} \bv_j }{1+ \eta \lambda_k \gamma_1^k} \\
J_o^{21} \bm{w}_1^h & = \frac{\eta  \frac{\alpha_2^k}{\beta_2^k}  \sum_{j\neq k } \lambda_j \xi_j^{h,1} \bv_j }{1+ \eta \lambda_k \gamma_2^k} \\
J_o^{22} \bm{w}_2^h & = \frac{\eta  \frac{\alpha_1^k}{\beta_2^k}  \sum_{j\neq k } \lambda_j \xi_j^{h,2} \bv_j }{1+ \eta \lambda_k \gamma_2^k} \\
\end{split}
\]
and
\[
\begin{split} 
J_d^{1} \bm{w}_1^h & = \frac{ \bw_1^h }{1+ \eta \lambda_k \gamma_1^k} \\
J_d^{2} \bm{w}_2^h & = \frac{ \bw_2^h  }{1+ \eta \lambda_k \gamma_2^k}
\end{split}
\]
The $ \bm{q}^h $ are obtained solving the algebraic equation  $ J(\pm \bv_k) \bm{q}^h = \mu \bm{q}^h $ where also the eigenvalue $ \mu $ is an unknown.
Expanding we obtain a block of $ n_1 $ identical equations
\[
\begin{split}
 \frac{  \sum_{j\neq k }  \left(  \eta n_ 1 \frac{\alpha_1^k}{\beta_1^k} \xi_j^{h,1}  +  \eta  n_2 \frac{\alpha_2^k}{\beta_1^k}   \xi_j^{h,2} \right)  \lambda_j  \bv_j +  \sum_{j\neq k }  \xi_j^{h,1}   \bv_j }{1+ \eta \lambda_k \gamma_1^k} \\
  = \mu  \sum_{j\neq k }  \xi_j^{h,1}   \bv_j 
 \end{split}
 \]
 and another of $ n_2 $ identical equations
 \[
 \begin{split}
 \frac{  \sum_{j\neq k }  \left( \eta n_ 1 \frac{\alpha_2^k}{\beta_2^k}  \xi_j^{h,1}  +  \eta  n_2 \frac{\alpha_1^k}{\beta_2^k}  \xi_j^{h,2} \right)  \lambda_j  \bv_j +  \sum_{j\neq k }  \xi_j^{h,2}   \bv_j }{1+ \eta \lambda_k \gamma_2^k} \\
  = \mu  \sum_{j\neq k }  \xi_j^{h,2}   \bv_j  \end{split}
 \]
 whose solution provides both the desired eigenvector $ \bm{w}_i^h $ and the associated eigenvalues $ \mu$. 
Projecting these equations along the eigenvector $ \bv_j $ and rearranging
\[
\begin{split}
\Big(  \overbrace{ 
 \frac{ 1+ \eta n_ 1 \frac{\alpha_1^k}{\beta_1^k}  \lambda_j  }{1+ \eta \lambda_k \gamma_1^k} }^{=a_j^k} - \mu \Big)  \xi_j^{h,1} + 
\Big(  \overbrace{ 
 \frac{ \eta n_2 \frac{\alpha_2^k}{\beta_1^k}  \lambda_j  }{1+ \eta \lambda_k \gamma_1^k} }^{=b_j^k} \Big)  \xi_j^{h,2} =0 \\
\Big( \underbrace{ \frac{ \eta n_1 \frac{\alpha_2^k}{\beta_2^k}   \lambda_j  }{1+ \eta \lambda_k \gamma_2^k}  }_{=c_j^k} \Big) \xi^{h,1}_j  
+ \Big( \underbrace{ \frac{  1+  \eta n_ 2 \frac{\alpha_1^k}{\beta_2^k}   \lambda_j  }{1+ \eta \lambda_k \gamma_2^k  } }_{=d_j^k} - \mu  \Big)  \xi^{h,2}_j  & = 0.
\end{split}
\]
Notice that $ a_j^k$, $b_j^k $, $ c_j^k$ and $ d_j^k $ are independent of the index $h$, hence, to avoid repeated identical algebraic equations, we can take $ h=j$ and drop one index in the $ \xi_j^{h,i}$ variables: $ \xi_j^{h,i} = \xi_j^i $, obtaining
\[
\begin{split}
(a_j^k-\mu) \xi^1_j  +b_j^k  \xi^2_j & =0 \\
c_j^k \xi^1_j  + (d_j^k-\mu)   \xi^2_j  & =0 ,
\end{split}
\]
which leads to the formula for the eigenvalues $ \mu_{j, \pm}^k = \frac{1}{2} \Big( a_j^k+ d_j^k \pm \sqrt{ (a_j^k-d_j^k)^2 + 4 c_j^kb_j^k } \Big) $. 
The relationship between the components of the two vectors $ \bm{w}_\ell^j $ is then $  \xi^2_j = -\frac{a_j^k- \mu_{j, \pm}^k}{b_j^k} \xi^1_j$.
Notice that for each $ \bm{q}^j $ there are two eigenvalues $ \mu_{j, \pm}^k$, for a total of $ 2 (d-1)$ eigenvalues in this class.

\item 
The third class contains the $ nd - n - 2d+2 $ remaining eigenvectors subdivided into two groups: $ \bm{r}_1 = \begin{bmatrix} \epsilon_1 \bv_h^T & \ldots &  \epsilon_{n_1} \bv_h^T & 0 & \ldots & 0 \end{bmatrix}^T$ and $ \bm{r}_2 = \begin{bmatrix} 0 & \ldots & 0  & \varepsilon_1 \bv_h^T & \ldots & \varepsilon_{n_2}  \bv_h^T \end{bmatrix}^T$ where $ \bv_h $ is an eigenvector of $\bV$, $h \neq k $, and $ \sum_{j=1}^{n_1} \epsilon_j =0$, $ \sum_{j=1}^{n_2} \varepsilon_j =0$.
By construction, $ \bv_k^T \bv_h =0 $ and $ (\bm{w}_\ell^j)^T \bv_h=0 $ for all  $ j=1, \ldots, k-1, k+1, \ldots, d $ and $ \ell =1,2$. 
Computing: 
\[
J(\pm \bv_k ) \bm{r}_1
= \begin{bmatrix}  
 \frac{\eta \frac{\alpha_1^k}{\beta_1^k} \lambda_h  \left(\sum_{j=1}^{n_1} \epsilon_j \right) \bv_h +  \epsilon_1 \bv_h}{ 1+ \eta \lambda_k \gamma_1^k} \\ 
\vdots \\ 
 \frac{\eta \frac{\alpha_1^k}{\beta_1^k} \lambda_h  \left(\sum_{j=1}^{n_1} \epsilon_j \right) \bv_h + \epsilon_{n_1} \bv_h}{ 1+ \eta \lambda_k \gamma_1^k} \\ 
\frac{\eta \frac{\alpha_2^k}{\beta_2^k} \lambda_h   \left(\sum_{j=1}^{n_1} \epsilon_j \right) \bv_h }{ 1+ \eta \lambda_k \gamma_2^k}  \\ 
\vdots \\
 \frac{\eta \frac{\alpha_2^k}{\beta_2^k} \lambda_h  \left(\sum_{j=1}^{n_1} \epsilon_j \right) \bv_h }{ 1+ \eta \lambda_k \gamma_2^k} 
 \end{bmatrix} = \frac{1}{1+ \eta \lambda_k \gamma_1^k}  \bm{r}_1
 \]  
and, similarly, $ J(\pm \bv_k ) \bm{r}_2 = \frac{1}{1+ \eta \lambda_k \gamma_2^k} \bm{r}_2$. 
\eenu
\qed

\begin{remark}
As for the multiagent Oja system, in Lemma~\ref{lem:lin-stab-self2} the first class of eigenvectors is ``extrinsic'' and the other two ``intrinsic'', i.e., they belong to the tangent space $ T_{\bx} (\mathbb{S}^{d-1})^n $. Only the last two classes count when deciding the local asumptotic stability properties of the bipartite equilibria. 
\end{remark}

\begin{lemma}
\label{lem:polyg-self}
The polygonal equilibria are all unstable for \eqref{eq:update}.
\end{lemma}
\proof
The idea of the proof is similar to that used in Lemma~\ref{lem:moja-polygonal-stab} for the multiagent Oja flow. 
A polygonal equilibrium point $ \bx =\bm{s} =\begin{bmatrix} \bm{s}_1 & \ldots & \bm{s}_n \end{bmatrix} \in (\mathbb{S}^{d-1})^n$ is s.t. $ \bV  \sum_{j=1}^n A_{ij}(\bm{s}) \bm{s}_j  =0$. Let us compute the linearization of \eqref{eq:update} at $\bm{s}$, obtained perturbing $ \bm{s}$ with a perturbation $ \bw = [ \bw_1 \ldots \bw_n] $ with  $  \bs_i^T \bw_i =0$.  For the non-normalized next iterate it is
\[
\begin{split}
\tilde{\bx}_i = & \bs_i+\bw_i + \frac{\eta}{n} \bV\sum_j A_{ij}(\bs + \bw) (\bs_j + \bw_j)  \\
= & \bs_i + \eta \underbrace{\bV\sum_j A(\bs) \bs_j}_{=0} + \bw_i + \eta  \bV \sum_j A(\bs) \bw_j \\
& +  \eta  \bV \sum_j \sum_h \pde{A_{ij}(\bs)}{\bw_h} \bw_h \bs_j + h.o.t.
\end{split}
\]
where the vectors $ \pde{A_{ij}(\bm{s})}{\bw_h} $ are computed in \eqref{eq:partial-A}.
Denote $ \psi^{ij} (\bw) := \sum_h\pde{A_{ij}(\bm{s})}{\bw_h}  \bw_h$ and observe that $ \psi^{ij} (\bw) $ is a sum of bilinear forms $ \bs_\ell B_k \bw_h $ for some matrices $ B_k $ (see \eqref{eq:partial-A} for their specific expressions). 
Since $ \| \bm{s}_\ell \|=1$ and $  \bw_h $ is small enough, from the matrix Cauchy-Schwartz inequality, for each of these bilinear forms it holds $ - \| B_k\|_2 \leq \bm{s}_\ell B_k \bw_h \leq \| B_k \|_2 $, where $ \|B_k \|_2 $ is  independent of $ \bw_h$.

Denoting $ \bu_i := \bw_i + \eta \bV  \sum_j \left( \psi_{ij}(\bw) \bs_j + A_{ij}(\bs) \bw_j \right) $, the norm of $ \tilde{\bx}_i $ is
\[
\| \tilde{\bx}_i \| = \Big( 1+ \| \bu_i \|^2 + 2 \eta \bs_i^T \bV \sum_j \left( \psi_{ij}(\bw) \bs_j + A_{ij}(\bs) \bw_j \right)  \Big)^{\frac{1}{2}}
\]
where we have used $ \bs_i^T \bw_i =0 $.
Observe that $ \| \bu_i \|^2 $ is composed entirely of terms of order 2 or higher in $ \bw_i $, hence, using the binomial series expansion ($ (1 + \xi)^{-\frac{1}{2}}\approx 1 - \frac{\xi}{2}  + h.o.t. $), 
\[
\begin{split}
& \bx_i^+ = \frac{\tilde{\bx}_i}{\| \tilde{\bx}_i \|}
\approx (\bs_i + \bu_i) \left(1- \frac{\| \bu_i \|^2}{2} \right. \\
&\qquad  \left. - \eta \bs_i^T \bV \sum_j \left( \psi_{ij}(\bw) \bs_j + A_{ij}(\bs) \bw_j \right) \right)  + h.o.t. \\
&\approx  \bs_i - \eta \bs_i^T \bV \sum_j \left( \psi_{ij}(\bw) \bs_j + A_{ij}(\bs) \bw_j \right) \bs_i  + \bu_i   + h.o.t.\\
& = \bs_i - \eta  \bs_i  \bs_i^T \bV \sum_j \left( \psi_{ij}(\bw) \bs_j + A_{ij}(\bs) \bw_j \right)+ \bw_i\\
& \qquad  + \eta \bV \sum_j \left( \psi_{ij}(\bw) \bs_j + A_{ij}(\bs) \bw_j \right) + h.o.t.\\
& = \bs_i + \bw_i +\eta (I - \bs_i  \bs_i^T ) \bV \sum_j \left( \psi_{ij}(\bw) \bs_j + A_{ij}(\bs) \bw_j \right).
\end{split}
\]
The linearized equation is then 
\[
\bw_i^+ = \bw_i +\eta (I - \bs_i  \bs_i^T ) \bV \sum_j \left( \psi_{ij}(\bw) \bs_j + A_{ij}(\bs) \bw_j \right).
\]
Expanding all $ \bw_i $ and $ \bs_i $ in the basis of eigenvectors $ \bv_1, \ldots, \bv_d $, and denoting $ \bw_i = \sum_\ell \xi^i_\ell \bv_\ell $ and $ \bs_i = \sum_\ell \zeta^i_\ell \bv_\ell $, 
\[
\begin{split}
\bw_i^+ & = \sum_\ell (\xi^i_\ell)^+ \bv_\ell =  \sum_\ell \xi^i_\ell \bv_\ell +\eta \Big( I - \sum_k \zeta^i_k \bv_k \sum_\ell \zeta^i_\ell \bv_\ell^T \Big) \cdot \\
& \quad   \cdot \sum_{j, \ell} \lambda_\ell (\psi_{ij}(\bw) \zeta^j_\ell + A_{ij}(\bs) \xi^j_\ell )  \bv_\ell .
\end{split}
\]

Assuming that the perturbation $ \bw$ is aligned with $ \bv_1$, i.e., for all $i$, $ \xi^i_1 \neq 0 $ and $ \xi^i_k=0$ for $ k=2, \ldots, d$, then, projecting along $ \bv_1$, yields
\[
{\xi^i_1}^+ =\xi^i_1 + \eta \left( 1- (\zeta^i_1)^2 \right) \lambda_1 \Big( \sum_j \psi_{ij}(\bw) \zeta^j_1 +  \sum_j A_{ij}(\bm{s}) \xi^j_1 \Big) .
\]
Denoting  $ \bm{\xi}_1 = \begin{bmatrix} \xi^1_1 & \ldots & \xi^n_1 \end{bmatrix}^T$
and $  \bm{\zeta}_1 = \begin{bmatrix} \zeta^1_1 & \ldots & \zeta^n_1 \end{bmatrix}^T$ the collection of the $ \xi_1^i $ and $ \zeta_1^i $ components of all agents, the previous ODEs can be expressed in vector form as
\beq
{\bm{\xi}_1}^+ =\bm{\xi}_1 + \eta  \lambda_1 \left( I- Z_1^2 \right) \left( \Psi \bm{\zeta}_1 + A(\bm{s} ) \bm{\xi}_1 \right)
\label{eq:ode-eta-self}
\eeq
where $ Z_1 = \diag(\bm{\zeta}_1)$, $ \Psi =[\psi_{ij}]$ is a matrix with lower and upper bounds independent of $ \bw$, and $ \bm{\zeta}_1 $ is fixed. $ A(\bm{s} ) $ is a row stochastic matrix,  $ \rho(A(\bm{s} )) =1 $ and $ A(\bm{s} ) \mathds{1} = \mathds{1}$. 
Since $ |\zeta^i_1 |< 1$ because  $ \bm{s} $ is not an eigenvector of $\bV$ and $ \lambda_1 >0$, it is $\rho(I +\eta \lambda_1(I-Z_1^2) A(\bs))>1 $, and the system \eqref{eq:ode-eta-self} diverges when $ \bm{\xi}_1 = \epsilon \mathds{1}$ for some scalar $ \epsilon$. Hence the polygonal equilibrium $ \bm{s}$ is unstable. 
\qed

\begin{remark}
As for the multiagent Oja system, the perturbation used in Lemma~\ref{lem:polyg-self} is``intrinsic'', i.e.,  $ \bw_i \in T_{\bs_i} \mathbb{S}^{d-1} $ for all $i$. 
\end{remark}

\noindent {\bf Proof of Theorem~\ref{thm:stab-main-self}.}
The proof is the direct combination of Lemmas~\ref{lem:lin-stab-self1},~\ref{lem:lin-stab-self2} and~\ref{lem:polyg-self}, with the only observation that in the second class of eigenvalues of the bipartite consensus case of Lemma~\ref{lem:lin-stab-self2}, using the Schur-Cohen formula for stability of a $ 2 \times 2 $ matrix in discrete time, the condition $ | \mu_{j, \pm}^k | <1$ becomes $  a_j^k d_j^k - b_j^k c_j^k < 1 $ and $ a_j^k + d_j^k <1+ a_j^k d_j^k - b_j^k c_j^k $, $\, \forall \, j=1, \ldots, k-1, k+1,\ldots,  d $.
\qed

\begin{remark}
\label{rem:bip-in-practice-singlehead}
Similarly to the continuous-time case analyzed in \cite{altafini2026multistability}, bipartite consensus equilibria can be locally stable for the single-head discrete-time transformer in \eqref{eq:update}, even though numerically this seems to be a rare event, and to be associated to very small basins of attraction. A simulation example is shown in Fig.~\ref{fig:3Dexamples-bip}.
Notice that under the assumption of Theorem~\ref{thm:stab-main-self}, from the third condition of the theorem ($ \frac{1}{1+ \eta \lambda_k \gamma_i^k} <1$), it is straightforward to deduce that if $ \lambda_k <0 $ is an eigenvalue of $ \bV $, then the associated bipartite equilibria $ \bx_i = \pm \bv_k $ are unstable, because it is $ 1+ \eta \lambda_k \gamma_i^k <1 $ for all admissible values of $ \eta< \min\left( \frac{1}{\lambda_1 } , \, \frac{1}{ | \lambda_d |} \right) $.
\end{remark}

\section{Extensions}
In this section we discuss a few extensions of the results, most notably to the multihead case.

\subsection{Multihead case}
\label{sec:ext-multi-head}

Real transformers are multihead, meaning that different pieces of the token vectors $ \bx_i $ attend to different convex combinations of the other tokens. 
This mechanism is constructed by using multiple query, key and value matrices, and combining them through a fourth weight matrix called output matrix, using an operation normally referred to as a ``concatenation'' \cite{vaswani2017attention}. 
Let $ h=1, \ldots, H $ be the head index. 
If $ W_Q^{(h)}$, $ W_K^{(h)}$, $ W_V^{(h)}$, and $ W_O$ are the learned matrices in the original notation of \cite{vaswani2017attention}, then in our column-token convention the corresponding matrices are obtained by transposition: 
\[
F_Q^{(h)} = (W_Q^{(h)})^T, \quad 
F_K^{(h)} = (W_K^{(h)})^T,\quad 
F_O = W_O^T
\]
\[
F_V^{(1)} = \begin{bmatrix} (W_V^{(1)})^T \\ 0 \\ \vdots \\ 0  \end{bmatrix}, \ldots,  F_V^{(H)} = \begin{bmatrix} 0 \\ 0 \\ \vdots \\ (W_V^{(H)})^T   \end{bmatrix},
 \]
where the zero-padding in $ F_V^{(h)}$ places the nonzero entries into the rows corresponding to the $ h$ head and corresponds to the ``concatenation'' operation of \cite{vaswani2017attention}.
Each head has an associated attention matrix $ A^{(h)}(\bx) $ of components 
\[
A_{ij}^{(h)}(\bx) = \frac{e^{\langle \bQ^{(h)} \bx_i, \bK^{(h)} \bx_j \rangle } }{\sum_{\ell=1}^n e^{\langle \bQ^{(h)} \bx_i, \bK^{(h)} \bx_\ell \rangle } } .
\]
The multi-head case corresponds to the following time-invariant dynamical system
\beq
\label{eq:update-multihead-auton-bis}
\begin{split}
\bx_i(t+1)  = \frac{\bx_i (t) +\eta F_{O} \sum_{h=1}^H F_{V}^{(h)}  \sum_j A_{ij}^{(h)}(\bx(t)) \bx_j (t)  }{\| \bx_i (t) +\eta F_{O} \sum_{h=1}^H F_{V}^{(h)}  \sum_j A_{ij}^{(h)}(\bx(t)) \bx_j  (t) \|}.
\end{split}
\end{equation}
As for the single head model, if we disregard attention (which varies from token to token) the matrix that lumps together the contribution of the $H$ heads occurring in the transformer layer is the following ``layer matrix''
\[
F = F_O  \sum_{h=1}^H F_V^{(h)} .
\] 
Let $ \lambda_1, \ldots, \lambda_d $ be the eigenvalues of $ F $, and $ \bv_1, \ldots, \bv_d $ the associated eigenvectors, normalized s.t. $ \| \bv_j\|=1$.
Under the special case of $ F $ which is symmetric, the asymptotic behavior of the system~\eqref{eq:update-multihead-auton-bis} is described in the following theorem.

\begin{theorem}
\label{thm:symm-case1}
Consider the system~\eqref{eq:update-multihead-auton-bis}. 
Assume the attention matrices $ A^{(h)}$, $ h=1, \ldots, H$, are full and the matrix $ F = F_O \sum_{h=1}^H F_{V}^{(h)} $ is symmetric with $ \lambda_1 > \lambda_i$, $ i=2,\ldots, d$, $ \lambda_1 >0$. Assume further that  $ \eta<  \frac{2}{ |\lambda_1 + \lambda_d |}  $ whenever $ - \lambda_d > \lambda_1 $. Then the consensus point $ \bx_i =\bv_1 $ $\forall \, i=1, \ldots, n$ (or $ \bx_i =-\bv_1 $ $\forall \, i=1, \ldots, n$) is a locally asymptotically stable equilibrium point for the system~\eqref{eq:update-multihead-auton-bis} when $ T\to \infty$. 
\end{theorem}

\proof Under the assumption of the theorem, the system~\eqref{eq:update-multihead-auton-bis} has the consensus points at the eigenvectors of $F$ as equilibria, as can be verified by an argument similar to Lemma~\ref{lem:self-equil1}. In particular, when $ \bx_i$ approaches a consensus point, the attention matrices of \eqref{eq:update-multihead-auton-bis} still approach a uniform matrix: $ A_{ij}^{(h)}(\bx) \to \frac{1}{n}$, $ h=1, \ldots, H$. Therefore we can use the same argument used in Theorem~\ref{thm:stab-main-self}: at consensus, the system \eqref{eq:update-multihead-auton-bis} has the Jacobian $ J(\bv_k) $ of the multiagent Oja system \eqref{eq:update-moja}, whose eigenvalues were computed in Lemma~\ref{lem:moja-F-eigenvalues}. Consensus aligned with the principal eigenvector $ \bv_1 $ is locally asymptotically stable, while all other consensus points at the other eigenvectors of $F$ are unstable. 
\qed

While the picture for consensus is easily obtainable from that of the single-head case, what happens to the bipartite consensus points is less clear. 
We currently have no (analytical nor numerical) evidence of such points being stable equilibria for \eqref{eq:update-multihead-auton-bis}.

\subsection{Longer step size}
To check stability of the consensus equilibria a condition on the step size is needed only when $ \lambda_d < 0$, and  $ | \lambda_d |> \lambda_1 $, i.e., when the spectral radius of $F$ (or $ \bV$ for the single-head case) is $ | \lambda_d | $, not $ \lambda_1 $. 
Extra conditions on $ \eta $ are required only for checking the stability of the bipartite consensus points. 
These extra conditions are only sufficient, and meant to simplify the expression of the eigenvalues of the Jacobian $ J(\pm \bv_k)$, see Lemma~\ref{lemm:simplif-eta}.
As mentioned above, locally asymptotically stable bipartite consensus equilibria appear to be rare, even in the single-head case.

\subsection{Asymmetric layer (or value) matrix}
\label{sec:asymm-A}
The assumption of symmetric $F $ (or $F_V$ in the single head case) can be significantly relaxed. 
In fact, what really matters to capture the asymptotic behavior of the self-attention dynamics is that the dominant eigenvalue of $F$, $ \lambda_1 $, is real, while $ \lambda_2, \ldots, \lambda_d $ can be any.
This together with a strict dominance condition on $ \lambda_1 $ (i.e., $ \lambda_1 > {\rm Re}(\lambda_j)$, $ j=2, \ldots, d $) is the essence of the Perron-Frobenius theorem extended beyond nonnegative matrices. 
In fact, when the dominance is strict  not much changes in the behavior of the single-head system~\eqref{eq:update} and neither of its multi-head version~\eqref{eq:update-multihead-auton-bis} by considering an $ F$ which is not symmetric.
A numerical analysis of the asymmetric $F$ case is carried out when dealing with the ALBERT transformer shown in Section~\ref{sec:ALBERT}.

\section{Experiments}

\subsection{Small-scale numerical examples}
\label{sec:extra-examples}
In Fig.~\ref{fig:3Dexamples} we show a numerical simulation of the dynamics of the single-head self-attention system \eqref{eq:update} when $ d=3 $ and $ n=10$.
In this case $ \bx_i\in \mathbb{S}^2 $, i.e., we can represent trajectories of the self-attention dynamics on a sphere in $ \mathbb{R}^3$. 
The trajectory in Fig.~\ref{fig:3Dexamples} converges to consensus, aligned with the principal eigenvector $ \bv_1 $ of $ F_V $. 

In Fig.~\ref{fig:3Dexamples-bip}, for the single-head case,  we show instead an example of a bipartite consensus equilibrium which is locally asymptotically stable, as checked analytically using Theorem~\ref{thm:stab-main-self}. In Fig.~\ref{fig:3Dexamples-bip}(a), a small perturbation added to a bipartite equilibrium point aligned with $ \bv_1 $ is reabsorbed. However, the basin of attraction is very small: as soon as a slightly larger perturbation is applied (Fig.~\ref{fig:3Dexamples-bip}(b)), the tokens on one of the two antipodal points migrate to the other point, and consensus is achieved (still aligned with $ \bv_1 $, as expected from Theorem~\ref{thm:stab-main-self}). 
The (very) small basin of attraction of the bipartite stable equilibria appears to be a constant characteristic of these equilibria in our simulations. Analytical calculations of the basin of attraction are prohibitive for this system.

\begin{figure}[ht]
     \centering    
      \begin{minipage}{6cm}
        \includegraphics[trim=2cm 0.5cm 2cm 1cm, clip=true,width=1\textwidth]{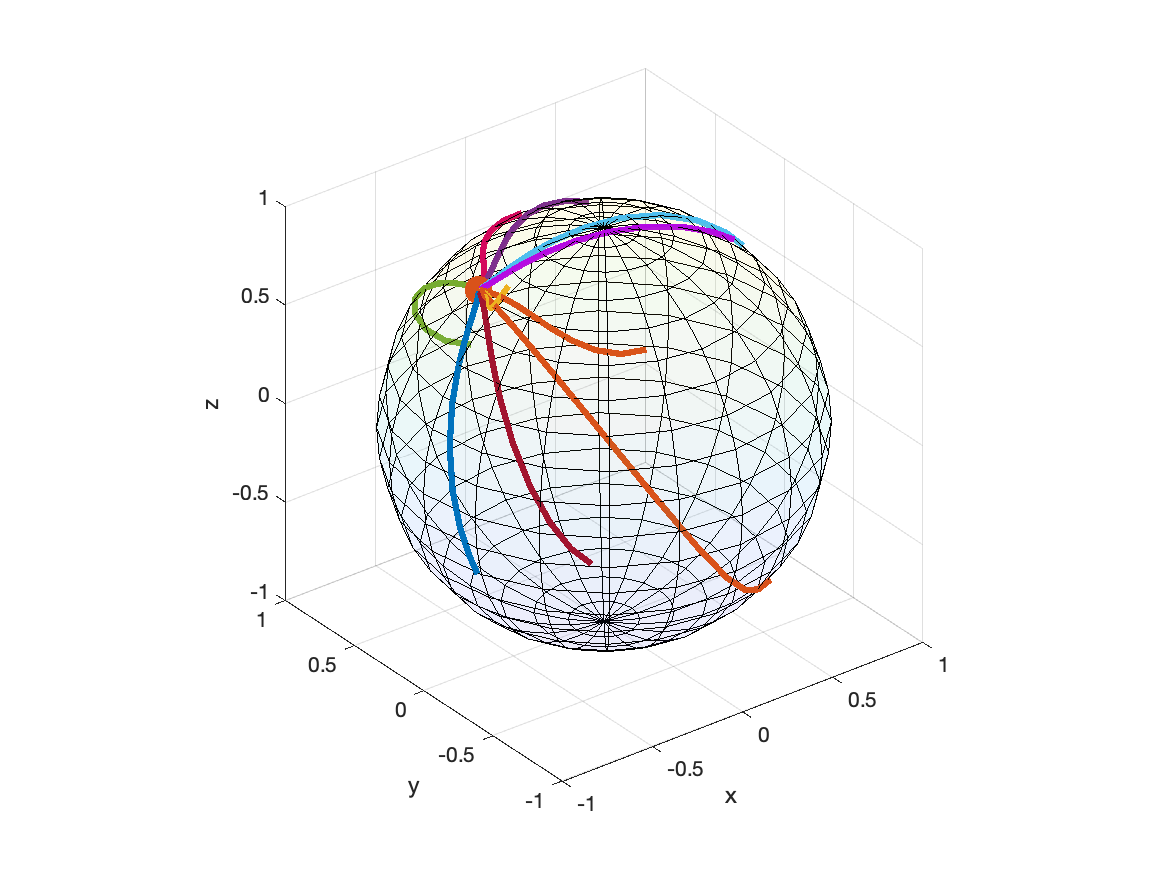}\\
        \includegraphics[trim=0cm 0cm 0cm 0cm, clip=true,width=1\textwidth]{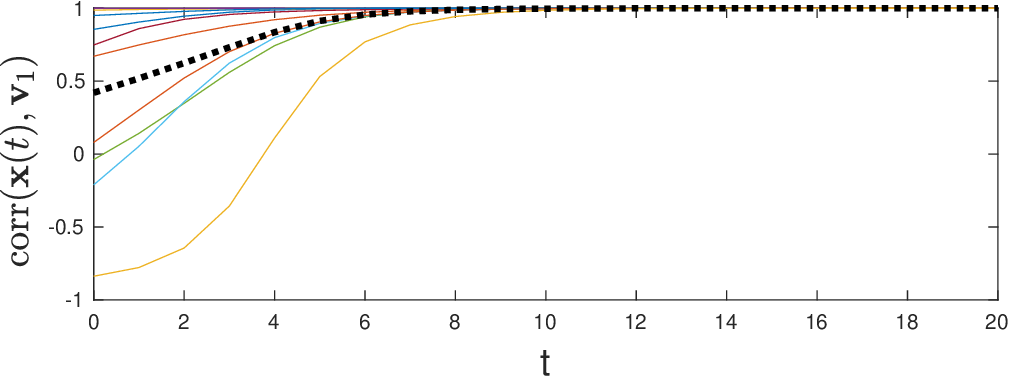}
        \end{minipage}
        \caption{Example of single-head self-attention dynamics with $ d=3$ and $ n=10$. Top panel: the trajectory of the 10 tokens is shown on the sphere $ \mathbb{S}^2 $. The solid dot is the endpoint of a trajectory, i.e., the consensus point reached by all tokens. Lower panel: the correlation $ {\rm corr}(\bx_i, \bv_1) $ tends to 1 as consensus is achieved (all tokens become aligned with the principal eigenvector $ \bv_1 $). The black dotted line is the mean correlation over the 10 tokens.}
        \label{fig:3Dexamples}
\end{figure}

\begin{figure}[ht]
     \centering    
   \subfigure[small perturbation preserves bipartition]{
      \begin{minipage}{6cm}
        \includegraphics[trim=2cm 0.5cm 2cm 1cm, clip=true,width=1\textwidth]{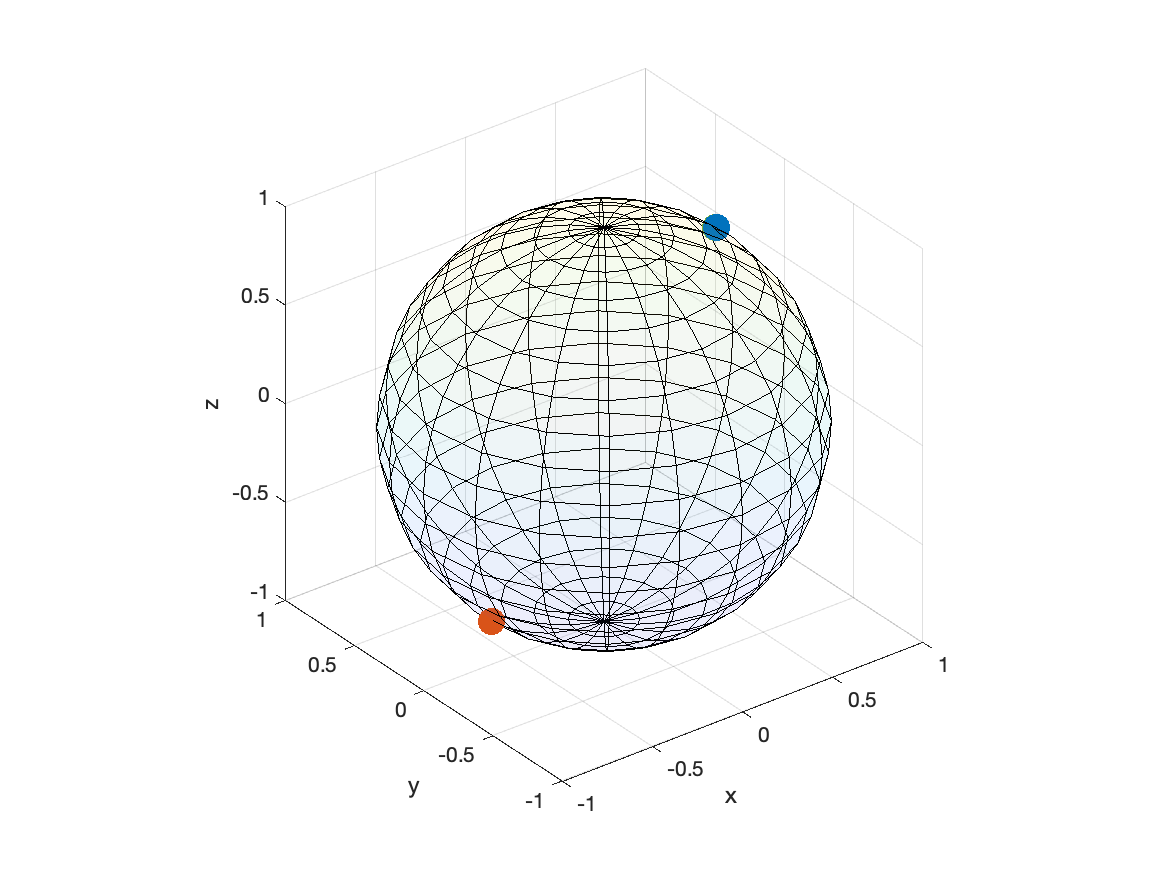}
      \includegraphics[trim=0cm 0cm 0cm 0cm, clip=true,width=1\textwidth]{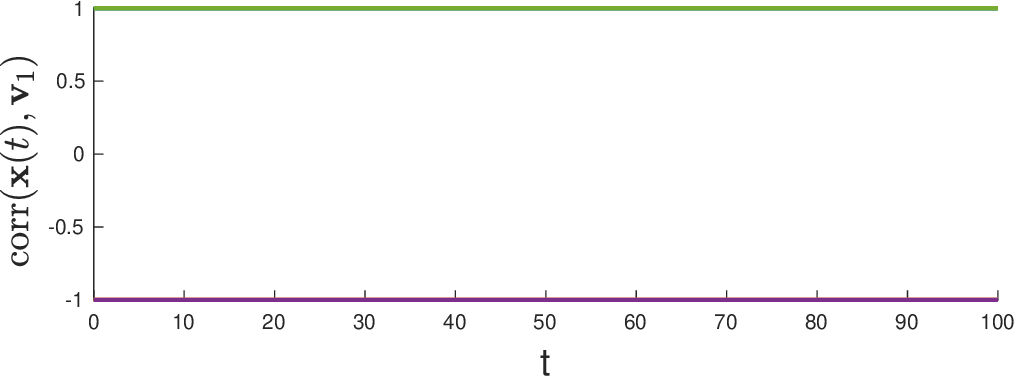}
      \end{minipage}}
      \subfigure[larger perturbation leads to consensus]{
      \begin{minipage}{6cm}
        \includegraphics[trim=2cm 0.5cm 2cm 1cm, clip=true,width=1\textwidth]{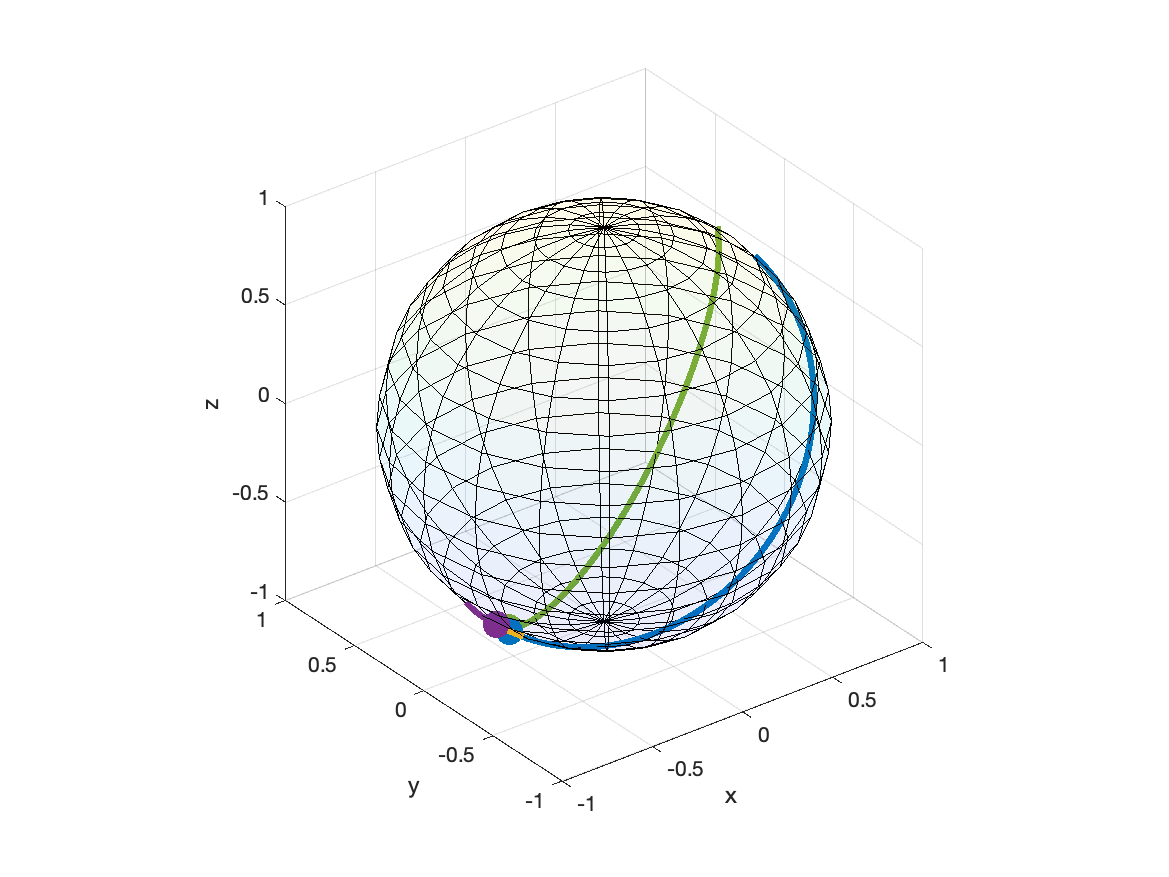}
        \includegraphics[trim=0cm 0cm 0cm 0cm, clip=true,width=1\textwidth]{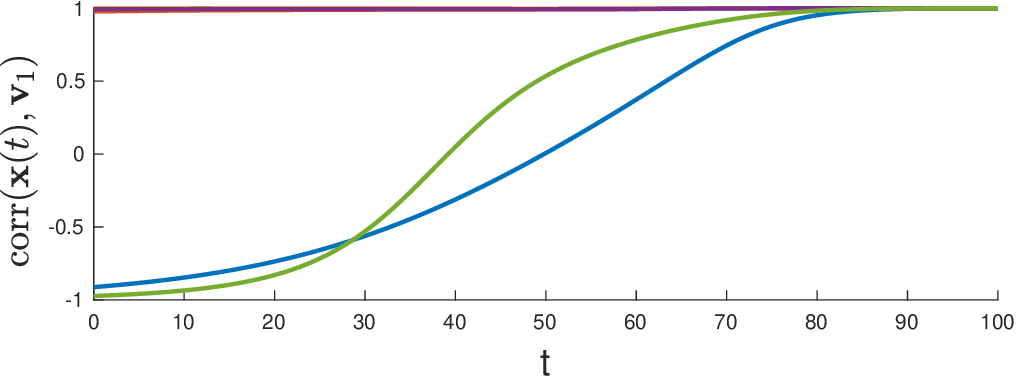}
     \end{minipage}}
        \caption{Example of bipartite consensus with $ d=3$ and $ n=10$. (a): when a very small perturbation is applied to a locally asymptotically stable bipartite consensus point (aligned with $ \bv_1 $), the system returns to the bipartite equilibrium. (b): when a larger perturbation is applied to the same bipartite consensus point, the trajectory escapes towards one of the two the consensus points associated to the principal eigenvector $ \bv_1 $. The two tokens in blue and green ``migrate'' to the antipodal point, thereby achieving consensus.
        }
        \label{fig:3Dexamples-bip}
\end{figure}

\subsection{ALBERT: a shared weight multihead transformer}
\label{sec:ALBERT}

ALBERT  \cite{lan2019albert} is a multihead transformer in which all weights are identical across layers.
The version we consider here, called ALBERT-large, is characterized by $ d=1024$, $H=16$ and $ T=24$. 
For it the principal eigenvalue of $F$, $ \lambda_1 =2.0674$, is real, positive, simple, and strictly dominant, hence the heuristic extension to nonsymmetric weight matrices mentioned in Section~\ref{sec:asymm-A} holds. 
Once we disregard the FFN part and we consider an $ L_2 $ normalization,  \eqref{eq:update-multihead-auton-bis} holds.

In the simulations we consider a small number of tokens $ n=10$ and the time horizon $T =24$ given by the real number of layers. 
Fig.~\ref{fig:ALBERT}(a) shows that over $ 10^3 $ experiments, the intertoken distance $ \| \bx_i - \bx_j \| $ tends to vanish rather quickly, meaning the tokens align with each other. 
Fig.~\ref{fig:ALBERT}(b) and~(c) show that the tokens align also with the principal eigenvector $ \bv_1 $ of $F$, even though this alignment is not complete within a time horizon equal to the number of layers ($T=24$). 
Consensus can be towards $ \bv_1 $ or $ - \bv_1 $, see Fig.~\ref{fig:ALBERT}(c).

\begin{figure}[ht]
 \centering    
   \subfigure[]{
      \begin{minipage}{6cm}
        \includegraphics[trim=0cm 7cm 0cm 0cm, clip=true,width=1\textwidth]{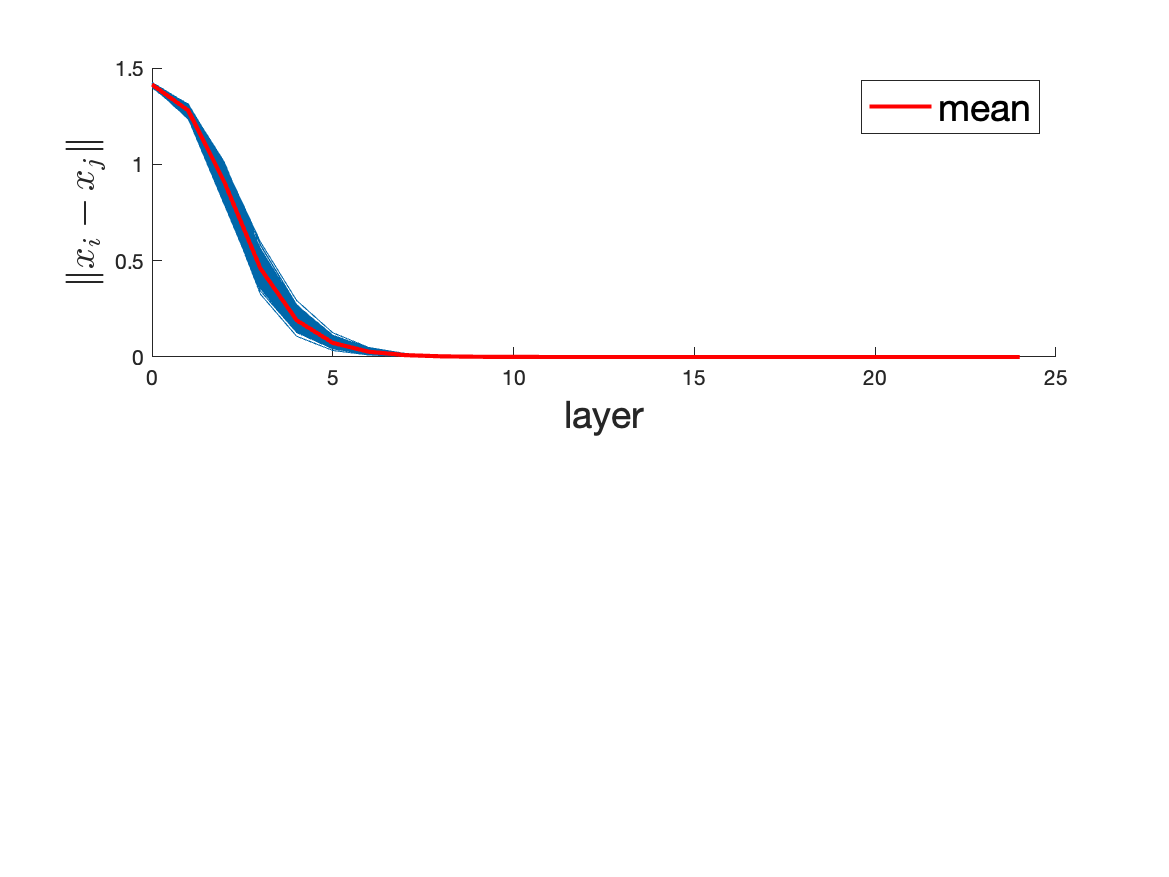}
     \end{minipage}}
   \subfigure[]{
      \begin{minipage}{6cm}
      \includegraphics[trim=0cm 0cm 0cm 0cm, clip=true,width=1\textwidth]{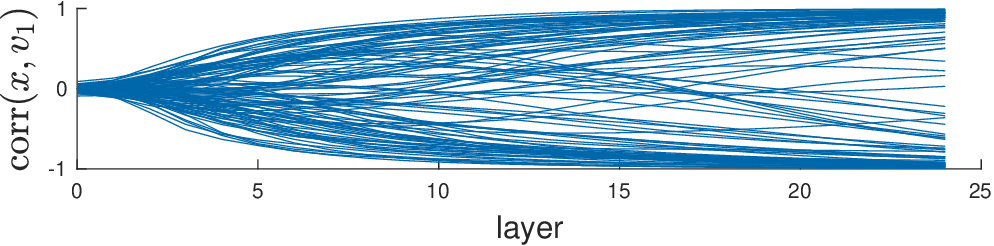}
      \end{minipage}}
   \subfigure[]{
      \begin{minipage}{6cm}
      \includegraphics[trim=0cm 7cm 0cm 0cm, clip=true,width=1\textwidth]{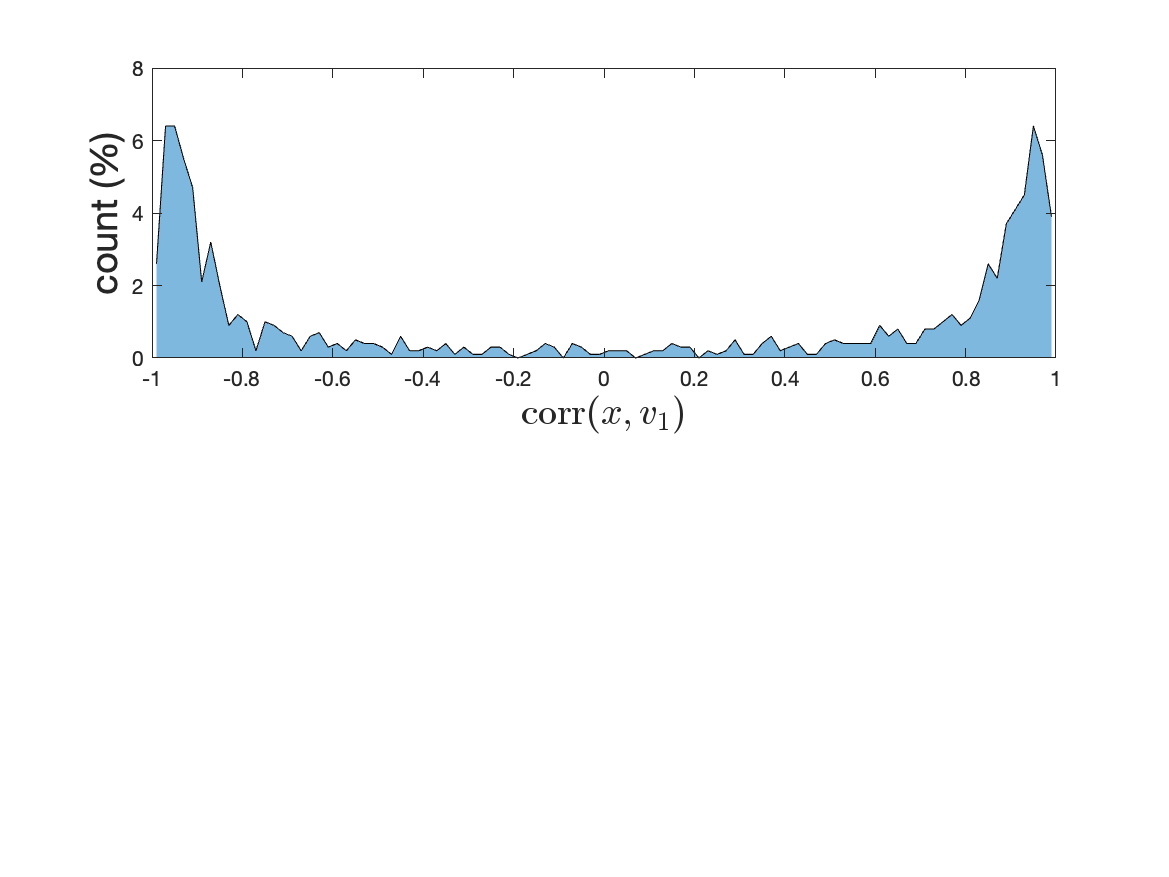}
      \end{minipage}}
        \caption{Example of real shared-weight transformer: ALBERT. (a): Inter-token distance for $ 10^3 $ experiments. (b): Correlation between $ \bx_i(t) $ and $ \bv_1 $ for $ t=1, \ldots, 24$ over multiple experiments. 
        (c): Histogram of the correlation between $ \bx_i(T) $ and $ \bv_1 $ at $T=24$ over $ 10^3 $ experiments. 
        }
        \label{fig:ALBERT}
\end{figure}

\section{Conclusions}
This paper provides a thorough analysis of the asymptotic behavior of a discrete-time self-attention dynamical system, meant to represent an idealized transformer. The locus of equilibria and their stability properties are investigated and confirmed to be analogous to their continuous-time counterparts. The analysis provides insight into the convergence properties of this model: the consensus point always aligns itself with the direction of the principal eigenvector of the value matrix, including in the multihead case (where the value matrix is replaced by what we call the layer matrix). These predictions are confirmed by investigating the behavior of a real shared weight transformer. 

 \section{Acknowledgments} The author would like to thank Chenglong Li for providing the weigh matrices of the ALBERT transformer. 


\end{document}